\documentclass{amsart} %{amsart}%{article}%{elsart}

\usepackage{Mytheorems}

\usepackage{amssymb, amsmath}
\usepackage{amsfonts}
\usepackage{blindtext}
\usepackage{color}
\usepackage{comment}
\usepackage{epsfig}
\usepackage{float}
\usepackage{graphicx}
\usepackage{array,makecell}
\usepackage{url}
\newtheorem*{thmA*}{Theorem A}
\newtheorem*{thmB*}{Theorem B}
\newtheorem*{lemA*}{Lemma A}

\newtheorem*{assumptionA*}{Assumption A}
\newtheorem*{assumptionB*}{Assumption B}
\newtheorem*{assumptionC*}{Assumption C}
\newtheorem{Winfree Model}{Winfree Model}

\newcommand{\R}{\mathbb R}

\newcommand{\dist}{\text{\rm dist}}
\def\be{\begin{equation}}
\def\ee{\end{equation}}

\def\R{\mathbb R}

\def\p{\partial}
\def\tilde{\widetilde}
\def\a{\alpha}

\renewcommand{\theequation}{\thesection.\arabic{equation}}
\numberwithin{figure}{section}
\vspace{1cm}

\def\R{\mathbb{R}}%the set of real numbers
\def\O{\Omega}% the bounded convex domain
\DeclareMathOperator{\diam}{{\rm diam}}%diameter function

\makeatletter
\newcommand{\opnorm}{\@ifstar\@opnorms\@opnorm}
\newcommand{\@opnorms}[1]{%
  \left|\mkern-1.5mu\left|\mkern-1.5mu\left|
   #1
  \right|\mkern-1.5mu\right|\mkern-1.5mu\right|
}
\newcommand{\@opnorm}[2][]{%
  \mathopen{#1|\mkern-1.5mu#1|\mkern-1.5mu#1|}
  #2
  \mathclose{#1|\mkern-1.5mu#1|\mkern-1.5mu#1|}
}
\makeatother

\allowdisplaybreaks

\begin{document}
\bibliographystyle{siam}

\title[]
{Geometric effects on $W^{1, p}$ regularity of the stationary linearized Boltzmann equation}

\author[I-K.~Chen, C.-H.~Hsia, D.~Kawagoe and J.-K.~Su]{I-KUN CHEN, CHUN-HSIUNG HSIA, DAISUKE KAWAGOE AND JHE-KUAN SU}

\date{\today}

\begin{abstract}
We study the incoming boundary value problem for the stationary linearized Boltzmann equation in bounded convex domains. The geometry of the domain has a dramatic effect on the space of solutions. We prove the existence of solutions in $W^{1,p}$ spaces for $1\leq p<2$ for small domains. In contrast, if we further assume the positivity of the Gaussian curvature on the boundary, we prove the existence of solutions in $W^{1, p}$ spaces for $1 \leq p < 3$ provided that the diameter of the domain is small enough. In both cases, we provide counterexamples in the hard sphere model; a bounded convex domain with a flat boundary for $p = 2$, and a small ball for $p = 3$.
\end{abstract}

\maketitle

\tableofcontents

\section{Introduction} \label{sec:intro}
It has been observed by many authors that the geometry of a domain has significant effects on the regularity of solutions to the Boltzmann equation. It has been observed that, for a non-convex domain, velocity grazing the boundary can enter the interior; while grazing only happens at the boundary for a convex domain. This difference makes a convex domain enjoy better regularity, see \cite{GKTT, GKTT2, Kimdis, KimLee} for time evolutional problems. We focus on the stationary equation on bounded convex domains. Surprisingly, a subtle difference in flatness of the boundary dramatically changes the range of solution spaces. 

We consider the following stationary linearized Boltzmann equation:
\begin{equation}\label{SLBE}
v \cdot \nabla_{x}f(x,v)=L(f)(x,v) \mbox{ for } (x,v) \in \Omega \times \mathbb{R}^{3}
\end{equation}
with the incoming boundary condition: 
\begin{equation}
\label{inbdry}
f(x,v)=g(x,v) \mbox{ for } (x,v) \in \Gamma^{-},
\end{equation}
where $\Gamma^{-}:=\lbrace (x,v)\in \partial \Omega \times \mathbb{R}^{3} \mid n(x)\cdot v<0 \rbrace$ and $n(x)$ is the outward unit normal of $\partial\Omega$ at $x$. Here, $\Omega$ is a bounded convex domain in $\R^3$ with $C^2$ boundary $\partial \Omega$ and $L$ is the linearized collision operator. The existence of solutions of lower regularity has been studied for linearized and nonlinear equations with various boundary conditions, see \cite{GuoKim,Gui 1,Gui 2}. As for regularity issue, based on the above existence results, the pointwise regularity is studied in \cite{RegularChen,CHK,chenH,chenkim}. 

A recent result \cite{I kun 1} studies the regularity in fractional Sobolev spaces. It proves that solutions are in $L^2(\mathbb{R}^3, H^s(\Omega))$ for $0 \leq s < 1$. Furthermore, the $L^2(\mathbb{R}^3, H^1(\Omega))$ regularity cannot be achieved by the formula of Bourgain, Brezis, and Mironescu in \cite{BBM}. It is natural to make a comparison with a famous regularity result \cite{GKTT}, in which the $W^{1,p}$ regularity is proved for the time evolutional full Boltzmann equation with the diffuse reflection boundary condition for $1 \leq p < 2$ and a counterexample for the $H^1$ solution is provided for the free transport equation therein. However, one should be aware that the results on the regularity issue of the evolution problem do not provide satisfactory information concerning the regularity properties of the stationary equation. For example, the Laplace equation can be regarded as the stationary equation of the wave equation while they have very different regularity properties. 

This motivates us the study on whether $H^1$ is the critical function space for regularity of the stationary Boltzmann equation. With the help of the smallness of the domain, in \cite{CCHKS} the authors prove existence of the $H^1$ solution for the stationary linearized Boltzmann equation with the incoming boundary condition for domains with boundary of positive Gaussian curvature, which gives a partial answer. In general, it is interesting to consider $W^{1, p}$ solutions for the Boltzmann equation. There is a study on the asymptotic stability of Boltzmann equation with regularity in space variable \cite{ChenKimGra}, which gives a hint on the range of $p$. In this article, we consider $W^{1, p}$ type solution spaces and classify the the range of the exponent $p$ according to the geometry of the domain.

We assume that the linearized collision operator $L$ satisfies the following assumption.

\begin{assumptionA*}
The operator $L(f)$ can be decomposed as
\[
L(f)(x, v) = - \nu(v) f(x, v) + K(f)(x, v), 
\]
where 
\begin{equation} \label{K}
K(f)(x,v):=\int_{\mathbb{R}^{3}}k(v,v^{*})f(x,v^{*})dv^{*}. 
\end{equation}
Here, we assume that $\nu(v)$ and $k(v, v^*)$ satisfy
\begin{align}
&\nu_{0}(1+|v|)^{\gamma} \leq \nu(v) \leq \nu_{1} (1+|v|)^{\gamma}, \label{AA}\\
&| k(v,v^{*}) | \lesssim \frac{1}{| v-v^{*} |(1+|v| +| v^{*} |)^{1-\gamma}} E_\rho(v, v^*), \label{AB}\\
&| \nabla_{v}k(v,v^{*}) | \lesssim \frac{1+|v|}{| v-v^{*} |^{2}(1+|v| +| v^{*} |)^{1-\gamma}}E_\rho(v, v^*), \label{AC}\\
&|\nabla_{v}\nu(v)|\lesssim(1+|v|)^{\gamma-1} \label{AD},
\end{align}
where $0 < \rho < 1$, $0 \leq \gamma \leq 1$, and
\begin{equation} \label{E}
E_\rho(v, v^*) := e^{-\frac{1-\rho}{4} \left( | v-v^{*} |^{2} + \left( \frac{|v|^{2}-| v^{*} |^{2}}{| v-v^{*} |} \right)^{2} \right)}.
\end{equation}
\end{assumptionA*}
Throughout this article, we adopt the convention $f \lesssim g$ if there exists a positive constant $C$ such that $f \leq Cg$. 

\begin{remark}
If we adopt the idea of Grad \cite{Grad} and consider the Grad angular cut-off potentials which include the hard sphere, hard potential, and Maxwellian molecular condition, then the condition of \eqref{AB} and the upper bound of \eqref{AA} hold. See \cite{Caf 1}. Also, it is worth mentioning that the commonly used cross section $B(|v-v^*|,\theta)= b|v-v^*|^{\gamma} \cos \theta$, where $b$ is a positive constant, leads to all the estimates in Assumption A.
\end{remark}

\begin{remark}
Instead of making use of \eqref{AA}--\eqref{AD} in {\bf Assumption A}, we shall employ the following estimates in the proof of our theorems:
\begin{align}
&\nu_{0} \leq \nu(v) \leq \nu_{1} (1+|v|), \label{AA'}\\
&| k(v,v^{*}) | \lesssim \frac{1}{| v-v^{*} |} E_\rho(v, v^*), \label{AB'}\\
&| \nabla_{v}k(v,v^{*}) | \lesssim \frac{1+|v|}{| v-v^{*} |^{2}}E_\rho(v, v^*), \label{AC'}\\
&|\nabla_{v}\nu(v)|\lesssim 1 \label{AD'}.
\end{align}
We notice that inequalities \eqref{AA'}--\eqref{AD'} can be derived from \eqref{AA}--\eqref{AD} in the case where $0 \leq \gamma \leq 1$.
\end{remark}

For the convenience of further discussion, we define
\begin{align*}
\tau(x, v) &:= \inf\lbrace s > 0 \mid x-sv\in\Omega^{c} \rbrace,\\
q(x,v) &:= x-\tau(x, v)v.
\end{align*}
With this notation, the equation \eqref{SLBE} can be expressed by the following integral form:
\begin{equation} \label{integral form}
\begin{split}
f =& e^{-\nu(v)\tau(x, v)}g(q(x,v),v)+\int_{0}^{\tau(x, v)}e^{-\nu(v)s}Kf(x-sv,v)ds\\
=&Jg+S_{\Omega}Kf,
\end{split}
\end{equation}
where
\begin{align}
Jg(x, v) :=& e^{-\nu(v)\tau(x, v)}g(q(x,v),v), \label{J}\\   
S_{\Omega}h(x, v) :=& \int_{0}^{\tau(x, v)}e^{-\nu(v)s}h(x-sv,v)ds. \label{S}
%K(f)(x,v):=& \int_{\mathbb{R}^{3}}k(v,v^{*})h(x,v^{*})dv^{*}.
\end{align}
Notice that we say $f$ is a solution to \eqref{SLBE} with boundary condition \eqref{inbdry} if $f$  satisfies \eqref{integral form}.

In order to state the main result, we introduce the function class of a solution space for the boundary value problem \eqref{SLBE}-\eqref{inbdry}. For $1 \leq p < \infty$ and $\a \geq 0$, let
\[
L^p_\a(\O \times \R^3) := \{ f \mid \| f \|_{L^p_\a(\O \times \R^3)} < \infty \},
\]
where
\[
\| f \|_{L^p_\a(\O \times \R^3)}^p := \int_{\O \times \R^3} |f(x, v)|^p e^{p \a |v|^2}\,dxdv.
\]
Also, for $1 \leq p < \infty$ and $\a \geq 0$, we define the function space $W^{1, p}_\a (\O \times \R^3)$ by
\[
W^{1, p}_\a (\O \times \R^3) := \{ f \mid \| f \|_{W^{1, p}_\a(\O \times \R^3)} < \infty \},
\]
where
\[
\| f \|_{W^{1, p}_\a(\O \times \R^3)} := \| f \|_{L^p_\a(\O \times \R^3)} + \| \nabla_x f \|_{L^p_\a(\O \times \R^3)} + \| \nabla_v f \|_{L^p_\a(\O \times \R^3)}.
\]
Notice that $W^{1, p}_\a(\O \times \R^3)$ with $\a = 0$ is the usual Sobolev space $W^{1, p}(\O \times \R^3)$. 

\begin{theorem} \label{main theorem}
Suppose the linearized collision operator $L$ satisfies {\bf Assumption A}. Let $\O$ be a bounded convex domain with $C^2$ boundary. Then, the following statements hold.

\begin{enumerate}
\item[(i)] For any given $1\leq p<2$ and $0 \leq \alpha < (1 - \rho)/2$, there exists $\epsilon = \epsilon(p, \a) >0$ such that: for any $\O$ with $\diam(\Omega)<\epsilon$, the boundary value problem \eqref{SLBE}-\eqref{inbdry} has a unique solution $f \in W^{1, p}_\alpha(\Omega \times \R^3)$ if and only if $Jg\in W^{1, p}_\a(\Omega \times \mathbb{R}^{3})$. \label{statement1}
 
\item[(ii)] If we further assume that $\partial\Omega$ is of positive Gaussian curvature,  then the range of $p$ in $\mathrm{(i)}$ can be extended to $1\leq p<3$. \label{statement2}
 
\item[(iii)] The conclusions in $\mathrm{(i)}$ and $\mathrm{(ii)}$ are optimal in the sense of Lemma \ref{lem:optimality2} and Lemma \ref{lem:optimality3}. \label{statement3}

\end{enumerate}
\end{theorem}

We technically assumed the smallness of the diameter of the domain which is equivalent to the case of large Knudsen number or the case of dilute gas. Thanks to this assumption, we can focus on discussing geometric effect on the $W^{1, p}$ estimate. We also emphasize that, though we put the Gaussian weight $e^{\a |v|^2}$ in the norms in order to describe the shape of solutions more precisely, it does not play an essential role in the $W^{1, p}$ estimate.

Here we briefly sketch the idea of the proof of Theorem \ref{main theorem}. A formal Picard iteration gives the following solution formula for \eqref{integral form}
\begin{equation} \label{Picard}
f=\sum_{i=0}^{\infty}(S_{\Omega}K)^{i}Jg,
\end{equation}
which is a valid solution of \eqref{integral form} if the right hand side of \eqref{Picard} converges in $W^{1, p}_\a(\O \times \R^3)$. 

We first consider convergence of \eqref{Picard} in $L^p_\a(\O \times \R^3)$. We introduce the following lemma.

\begin{lemma}\label{estimate on Lp}
Let $1 \leq p < \infty$ and $0 \leq \alpha < (1 - \rho)/2$, where $\rho$ is the constant in {\bf Assumption A}. Then, for any $h \in L^p_\a(\Omega \times \mathbb{R}^{3})$, we have
\begin{equation}
\| S_{\Omega}Kh \|_{L^p_\a(\Omega \times \mathbb{R}^{3})} \lesssim \diam(\Omega)^{\frac{1}{p}} \| h \|_{L^p_\a(\Omega \times \mathbb{R}^{3})}.
\end{equation}
\end{lemma}

By Lemma \ref{estimate on Lp}, we can prove that $S_{\Omega}K$ is a contraction mapping in $L^p_\a(\O \times \R^3)$ provided the diameter of $\Omega$ is sufficiently small, which leads the $L^p_\a$ convergence of \eqref{Picard}. However, concerning the $W^{1, p}_\a$ convergence of \eqref{Picard}, we do not have a direct analogy of Lemma \ref{estimate on Lp}. Instead, for $1 \leq p < 2$, we have the following estimate.

\begin{lemma} \label{estimate on W1p for 1 to 2} 
Given $h \in W^{1, p}_\a (\Omega \times \mathbb{R}^{3})$ with $1 \leq p < 2$ and $0 \leq \a < (1 - \rho)/2$, where $\rho$ is the constant in {\bf Assumption A}, we have
\begin{align*}
\| S_{\Omega}Kh \|_{W^{1, p}_\a(\Omega \times \mathbb{R}^{3})} \lesssim& \diam(\Omega)^{\frac{1}{p}} \| h \|_{W^{1, p}_\a(\Omega \times \mathbb{R}^{3})} + \| h\|_{L^p_\a(\O \times \R^3)}\\
&+ \diam(\O)^{\frac{1}{p}} \| h \|_{L^p_\a(\partial \O \times \R^3)},
\end{align*}
where $\| h \|_{L^p_\a(\partial\Omega \times \mathbb{R}^{3})}$ is defined as
\[
\| h \|_{L^p_\a(\partial\Omega\times \mathbb{R}^{3})}^p:=\int_{ \partial \Omega \times \mathbb{R}^{3}} |h(z,v)|^p e^{p \a |v|^2}\,d\Sigma(z)dv,
\]
and $d\Sigma$ denotes the surface measure on $\p \O$.
\end{lemma}

To control the boundary integral term in the estimate in Lemma \ref{estimate on W1p for 1 to 2}, we introduce following trace inequalities. 

\begin{lemma}\label{BLp_est}
Let $\O$ be a bounded domain with Lipschitz boundary. Also, let $1 \leq p < \infty$ and $\a \geq 0$. Then, there exists a positive constant $C_1(\O)$ such that
\[
\| h \|_{L^p_\a(\partial \Omega \times \mathbb{R}^{3})} 
\leq C_1(\O) \left( \delta^{\frac{p - 1}{p}} \| \nabla_x h \|_{L^p_\a(\Omega \times \mathbb{R}^{3})} + \delta^{-\frac{1}{p}}\| h \|_{L^p_a(\O \times \R^3)} \right)
\]
for all $h \in W^{1, p}_\a(\O \times \R^3)$ and $0 < \delta < 1$. 
\end{lemma}

\begin{lemma} \label{BL1_est}
Let $\O$ be a bounded domain with $C^2$ boundary, and let $\a \geq 0$. Then, for any $\delta>0$, there exists a positive constant $C_\delta(\O)$ such that
\[
\| h \|_{L^1_\a(\p \O \times \R^3)} \leq (1 + \delta) \| \nabla_x h \|_{L^1_\a(\O \times \R^3)} + C_\delta(\O) \| h \|_{L^1_\a(\O \times \R^3)}
\]
for all $h \in W^{1, 1}_\a(\O \times \R^3)$.
\end{lemma}

For fixed $1 \leq p < 2$ and $0 \leq \a < (1 - \rho)/2$, combining Lemma \ref{estimate on W1p for 1 to 2} with Lemma \ref{BLp_est} and Lemma \ref{BL1_est}, and taking $\delta$ and $\diam(\O)$ sufficiently small, we have
\begin{equation} \label{SB13}
\begin{split}
\| (S_{\Omega}K)^i Jg \|_{W^{1, p}_\a(\Omega \times \mathbb{R}^{3})} \leq& \frac{1}{2} \| (S_{\Omega}K)^{i-1} Jg \|_{W^{1, p}_\a(\Omega \times \mathbb{R}^{3})}\\
&+ C_2 \| (S_{\Omega}K)^{i-1} Jg\|_{L^p_\a(\O \times \R^3)},
\end{split}
\end{equation}
where $C_2 = C_2(\O, p, \a)$ is some positive constant depending on $\Omega$, $p$ and $\a$. Taking Lemma \ref{estimate on Lp} into consideration, the summation of \eqref{SB13} shows the convergence of \eqref{Picard} in $W^{1, p}_\a(\O \times \R^3
)$. This is the strategy of the proof of the statement (i) in Theorem \ref{main theorem}. 

For the case $2 \leq p < 3$, we need to use a good property of positive Gaussian curvature. We recall the following estimate from \cite{I kun 1}.

\begin{lemma}[Proposition 5.9 in \cite{I kun 1}] \label{circle lemma}
Let $\O$ be a $C^2$ bounded convex domain of positive Gaussian curvature. Then, there exists a positive constant $C_3(\O)$ depending only on $\Omega$ such that for any $(z, v) \in \Gamma^-$ we have 
\[
| z-q(z,-v) | \leq C_3(\O)N(z,v),
\]
where
\[
N(z, v) := |n(z) \cdot \hat{v}|, \quad \hat{v} := \frac{v}{|v|}.
\]
\end{lemma}

By Lemma \ref{circle lemma}, we have the following estimate.

\begin{lemma} \label{estimate on W1p for 2 to 3} 
Let $\O$ be a $C^2$ bounded convex domain of positive Gaussian curvature, and let $C_3(\O)$ be a constant defined in Lemma \ref{circle lemma}. Then, given $h \in W^{1, p}_\a (\Omega \times \mathbb{R}^{3})$ with $2 \leq p < 3$ and $0 \leq \a < (1 - \rho)/2$, where $\rho$ is the constant in {\bf Assumption A}, we have
\begin{align*}
\| S_{\Omega}Kh \|_{W^{1, p}_\a(\Omega \times \mathbb{R}^{3})} \lesssim& \diam(\Omega)^{\frac{1}{p}} \| h \|_{W^{1, p}_\a(\Omega \times \mathbb{R}^{3})} + \| h\|_{L^p_\a(\O \times \R^3)}\\
&+ C_3(\O)^{\frac{1}{p}} \| h \|_{L^p_\a(\partial\Omega \times \mathbb{R}^{3})}.
\end{align*}
\end{lemma}

For fixed $2 \leq p < 3$ and $0 \leq \a < (1 - \rho)/2$, combining Lemma \ref{estimate on W1p for 2 to 3} with Lemma \ref{BLp_est} and taking $\delta$ in Lemma \ref{BLp_est} and $\diam(\O)$ sufficiently small, we have \eqref{SB13} with some positive constant $C_2$ depending on $\Omega$, $p$ and $\a$. From \eqref{SB13} with Lemma \ref{estimate on Lp}, we see the convergence of the series \eqref{Picard} in $W^{1, p}_\a(\O \times \R^3)$. This is the proof of the statement (ii) in Theorem \ref{main theorem}.

Precise statements for (iii) in Theorem \ref{main theorem} are as follows.

\begin{lemma} \label{lem:optimality2}
We consider the hard sphere case $\gamma=1$. For fixed $1 \leq p < 2$ and $0 \leq \a < (1 - \rho)/2$, there exist a bounded convex domain $\Omega$ and a boundary data $g$ such that the boundary value problem \eqref{SLBE}--\eqref{inbdry} has a solution in $L^2(\O \times \R^3) \cap W^{1, p}_\a(\O \times \R^3)$ but this solution does not belong to $W^{1, 2}_\a(\O \times \R^3)$.
\end{lemma}

\begin{lemma} \label{lem:optimality3}
We consider the hard sphere case $\gamma=1$. For fixed $2 \leq p < 3$ and $0 \leq \a < (1 - \rho)/2$, there exist a bounded convex domain $\Omega$ with its boundary of positive Gaussian curvature and a boundary data $g$ such that the boundary value problem \eqref{SLBE}--\eqref{inbdry} has a solution in $L^3(\O \times \R^3) \cap W^{1, p}_\a(\O \times \R^3)$ but this solution does not belong to $W^{1, 3}_\a(\O \times \R^3)$.
\end{lemma}

Concerning Lemma \ref{lem:optimality2}, we choose $\Omega$ as a small bounded convex domain such that
\begin{equation} \label{flat_boundary}
D_{r_1} := \{ x = (0, x_2, x_3) \in \R^3 \mid |x| < r_1 \} \subset \partial \Omega
\end{equation}
with a small radius $r_1$ and 
\begin{equation} \label{halfball_contained}
\{ x = (x_1, x_2, x_3) \in \R^3 \mid |x| < r_1, x_1 < 0 \} \subset \Omega.
\end{equation}
We remark that $n(0) = (1, 0, 0)$. Let $\varphi_1$ be a smooth cut-off function on $\partial \O$ such that $0 \leq \varphi_1 \leq 1$, $\varphi_1(x) = 1$ for $x \in D_{r_1/4}$, and $\varphi_1(x) = 0$ for $x \in \partial \O \setminus D_{r_1/2}$. We pose the boundary data $g$ of the form: 
\begin{equation} \label{inbdry_optimal_convex_2}
g(x, v) = \varphi_1(x) e^{-\frac{1}{2} |v|^2}, \quad (x, v) \in \Gamma^-.
\end{equation}
We shall see in Section \ref{sec:counterexample} that, when $\diam(\O)$ is small enough, the solution to the boundary value problem \eqref{SLBE}-\eqref{inbdry} with the boundary data \eqref{inbdry_optimal_convex_2} is an example of Lemma \ref{lem:optimality2}.

For Lemma \ref{lem:optimality3}, the domain we consider is a small ball centered at the origin with radius $r$. We introduce the spherical coordinates on the boundary: $x = (r \cos \theta, r \sin \theta \cos \phi, r \sin \theta \sin \phi)$ for $\theta \in [0, \pi]$ and $\phi \in [0, 2 \pi)$. With these coordinates, for $\theta_0 \in (0, \pi)$, let $\partial \O_{\theta_0} := \{ x \in \partial \O \mid 0 \leq \theta < \theta_0 \}$. Take $0 < \theta_1 < \theta_2 < \pi$ and a smooth cut-off function $\varphi_2$ on $\partial \O$ such that $\varphi_2(x) = 1$ for $x \in \partial \O_{\theta_1}$, $\varphi_2(x) = 0$ for $x \in \partial \O \setminus \partial \O_{\theta_2}$, and $0 \leq \varphi_2(x) \leq 1$ for $x \in \partial \O_{\theta_2} \setminus \partial \O_{\theta_1}$. We pose the boundary data $g$ of the form: 
\begin{equation} \label{inbdry_optimal_convex_3}
g(x, v) = \varphi_2(x) e^{-\frac{1}{2} |v|^2}, \quad (x, v) \in \Gamma^-.
\end{equation}
We shall also see in Section \ref{sec:counterexample} that, if we choose $r$, $\theta_2$ and $\theta_1$ small enough, the solution to the boundary value problem \eqref{SLBE}-\eqref{inbdry} with the boundary data \eqref{inbdry_optimal_convex_3} is an example of Lemma \ref{lem:optimality3}.

The organization of the rest part of this article is as follows. 
In Section \ref{sec:pre}, we prepare some preliminary estimates for the operators $J$, $S_\O$ and $K$ on $L^p_\a$ spaces and give a proof of Lemma \ref{estimate on Lp}. Also, we introduce weighed spaces of bounded continuous functions, which are denoted by $C_\a$, and show boundedness of these operators on these spaces. In addition, we introduce a contraction mapping argument on $C_\a$, which will be used in Section \ref{sec:counterexample}. 
In Section \ref{sec:bounded_domain}, by showing Lemma \ref{estimate on W1p for 1 to 2}, Lemma \ref{BLp_est} and Lemma \ref{BL1_est}, we derive the estimate \eqref{SB13} for fixed $1 \leq p < 2$ and $0 \leq \a < (1 - \rho)/2$. 
In Section \ref{sec:bounded_domain_B}, we further assume the positivity of Gaussian curvature to improve the estimate in Lemma \ref{estimate on W1p for 1 to 2} and obtain Lemma \ref{estimate on W1p for 2 to 3}. The improved estimate shows that the estimate \eqref{SB13} holds for $2 \leq p < 3$ and $0 \leq \a < (1 - \rho)/2$. 
In Section \ref{sec:counterexample}, we prove Lemma \ref{lem:optimality2} and Lemma \ref{lem:optimality3}.

%%%%%%%%%%%%%%%%%%%%%%%%%%%%%%%%%%%%%%%%%%%%%%%%%%%%%%%%%%%%%%%%%%%%%%%%%%%%%%%%%%%%%%%%%%%%%%%%%%%%%%%
\section{Preliminaries} \label{sec:pre}
%%%%%%%%%%%%%%%%%%%%%%%%%%%%%%%%%%%%%%%%%%%%%%%%%%%%%%%%%%%%%%%%%%%%%%%%%%%%%%%%%%%%%%%%%%%%%%%%%%%%%%%

In this section, we introduce some important estimates as preliminaries.

Let $\alpha \geq 0$ and $1 \leq p < \infty$. Define a function space $L^p_\alpha(\Gamma^-)$ for the incoming boundary data by
\[
L^p_\alpha(\Gamma^-) := \left\{ g \in L^p(\Gamma^-) \mid \| g \|_{L^p_\alpha(\Gamma^-)} < \infty \right\},
\]
where 
\begin{align*}
\| g \|_{L^p_\alpha(\Gamma^-)}^p :=& \int_{\Gamma^-} |g(z, v)|^p e^{p \alpha |v|^2} N(z, v) |v|\,d\Sigma(z) dv,
\end{align*}
and $N(z, v)$ is the function defined in Lemma \ref{circle lemma}. We remark that, unlike the norm $\| \cdot \|_{L^p_\a(\partial \O \times \R^3)}$, the norm $\| \cdot \|_{L^p_\alpha(\Gamma^-)}$ has the extra weight $N(z, v) |v|$. 

We first give an estimate for the boundedness of the operator $J$.

\begin{proposition} \label{prop:bound_J_alpha}
Let $\O$ be a bounded convex domain. Let $J$ be the operator as defined by \eqref{J}. Also, let $\a \geq 0$ and $1 \leq p < \infty$. Then, for all $g \in L^p_\alpha(\Gamma^-)$, we have $Jg \in L^p_\alpha(\O \times \R^3)$. Moreover, we have
\[
\opnorm{J}_{p, \alpha} := \sup_{\substack{g \in L^p_\alpha(\Gamma^-) \\ g \neq 0}} \frac{\| Jg \|_{L^p_\alpha(\O \times \R^3)}}{\| g \|_{L^p_\alpha(\Gamma^-)}} \leq \frac{1}{(p \nu_0)^{\frac{1}{p}}},
\]
where $\nu_0$ is the constant in \eqref{AA}.
\end{proposition}

\begin{proof}
By the change of coordinates: $x = z + tv$ with $(z, v) \in \Gamma^-$ and $t > 0$, we have
\begin{equation} \label{domain_to_boundary}
\begin{split}
\int_{\O \times \R^3} |f(x, v)|\,dxdv =& \int_{\R^3} \int_{\Gamma^-_v} \int_0^{\tau(z, - v)} |f(z + tv, v)|\,dt\,N(z, v) |v|\,d\Sigma(z) dv\\
=& \int_{\partial \O} \int_{\Gamma^-_z} \int_0^{\tau(z, -v)} |f(z + tv, v)|\,dt\,N(z, v) |v|\,dv d\Sigma(z)
\end{split}
\end{equation}
for $f \in L^1(\O \times \R^3)$, where
\[
\Gamma^-_v := \{ z \in \p \O \mid n(z) \cdot v < 0 \}
\]
for fixed $v \in \R^3$ and
\[
\Gamma^-_z := \{ v \in \R^3 \mid n(z) \cdot v < 0 \}
\]
for fixed $z \in \p \O$. See \cite{CS} for the detail. Thus, for $g \in L^p_\a(\O \times \R^3)$, we have
\begin{align*}
\| J g \|_{L^p_\a(\O \times \R^3)}^p =& \int_{\O \times \R^3} e^{- p \nu(v) \tau(x, v)} |g(q(x, v), v)|^p e^{p \a |v|^2}\,dxdv\\
=& \int_{\R^3} \int_{\Gamma^-_v} \int_0^{\tau(z, -v)} e^{-p \nu(v) t} |g(z, v)|^p\,dt N(z, v) |v| e^{p \a |v|^2}\,d\Sigma(z) dv\\
\leq& \frac{1}{p \nu_0} \int_{\R^3} \int_{\Gamma^-_v} |g(z, v)|^p N(z, v) |v| e^{p \a |v|^2}\,d\Sigma(z) dv\\
=& \frac{1}{p \nu_0} \| g \|_{L^p_\a(\Gamma_-)}^p,
\end{align*}
which implies that $\| Jg \|_{L^p_\a(\O \times \R^3)} \leq (p \nu_0)^{-1/p} \| g \|_{L^p_\a(\Gamma_-)}$. This completes the proof.
\end{proof}

We next introduce estimates for integrals which are related to the operator $S_\O$ and its derivatives.

\begin{lemma} \label{lem:est_S_convex}
Let $\O$ be a bounded convex domain. Then, for $v \neq 0$ and $a > 0$, we have
\[
\int_0^{\tau(x, v)} e^{-a t}\,dt \leq \min \left\{ \frac{1}{a}, \frac{\diam(\O)}{|v|} \right\}
\]
for all $x \in \O$.
\end{lemma}

\begin{proof}
By the direct integration, we have
\[
\int_0^{\tau(x, v)} e^{-a t}\,dt = \frac{1}{a} \left( 1 - e^{- a \tau(x, v)} \right) \leq \frac{1}{a}.
\]
On the other hand, since $e^{- a t} \leq 1$, we have
\[
\int_0^{\tau(x, v)} e^{-a t}\,dt \leq \tau(x, v) \leq \frac{\diam(\Omega)}{|v|}. 
\]
The last estimate follows from the definition of the function $\tau(x, v)$. This completes the proof.
\end{proof}

\begin{corollary} \label{cor:est_S_convex}
Let $\O$ be a bounded convex domain. Then, for $v \neq 0$, $a > 0$ and $b \geq 0$, we have
\[
\int_0^{\tau(x, v)} t^b e^{-a t}\,dt \lesssim \min \left\{ 1, \frac{\diam(\O)}{|v|} \right\}
\]
for all $x \in \O$.
\end{corollary}

\begin{proof}
The conclusion follows from the estimate
\[
t^b e^{-a t} \leq \left( \sup_{t > 0} t^b e^{-\frac{a}{2} t} \right) e^{-\frac{a}{2} t} \lesssim e^{-\frac{a}{2} t}.
\]
\end{proof}

We give an estimate for the boundedness of the operator $S_\O$.

\begin{proposition} \label{prop:bound_S_alpha}
Let $\O$ be a bounded convex domain. Let $S_\Omega$ be the operator as defined by \eqref{S}. Also, let $\a \geq 0$ and $1 \leq p < \infty$. Then, for all $f \in L^p_\alpha(\O \times \R^3)$, we have $S_\O f \in L^p_\alpha(\O \times \R^3)$. Moreover, we have
\[
\opnorm{S_\Omega}_{p, \alpha} := \sup_{\substack{f \in L^p_\alpha(\O \times \R^3) \\ f \neq 0}} \frac{\| S_\O f \|_{L^p_\alpha(\O \times \R^3)}}{\| f \|_{L^p_\alpha(\O \times \R^3)}} \leq \frac{1}{\nu_0},
\]
where $\nu_0$ is the constant in \eqref{AA}.
\end{proposition}

To prove Proposition \ref{prop:bound_S_alpha}, we shall use the following lemma, which is proved in \cite{CCHKS}. 

\begin{lemma}\label{change of variable lemma}
Let $\O$ be a bounded convex domain. Then, for a nonnegative measurable function $h$ on $\O \times \R^3 \times [0, \infty)$, we have
\[
\int_{\mathbb{R}^3}\int_{\Omega}\int_0^{\tau(x, v)}h(x,v,s) \,dsdxdv=\int_{\mathbb{R}^3}\int_{\Omega}\int_0^{\tau(y, -u)}h(y+tu,u,t) \,dtdydu
\]    
\end{lemma}

\begin{proof}[Proof of Proposition \ref{prop:bound_S_alpha}]
For $f \in L^p_\a(\O \times \R^3)$, we have
\[
\| S_\O f \|_{L^p_\a(\O \times \R^3)}^p = \int_{\O \times \R^3} \left| \int_0^{\tau(x, v)} e^{- \nu(v) s} f(x - sv, v)\,ds \right|^p\, e^{p \a |v|^2} \,dxdv.
\]

When $p = 1$, by setting $h(x, v, s) = e^{-\nu_0 s} |f(x - sv, v)|$ and applying Lemma \ref{change of variable lemma}, we obtain
\begin{align*}
&\int_{\O \times \R^3} \left| \int_0^{\tau(x, v)} e^{- \nu(v) s} f(x - sv, v)\,ds \right| \,e^{\a |v|^2} \,dxdv\\
\leq& \int_{\O \times \R^3} \int_0^{\tau(x, v)} e^{- \nu_0 s} |f(x - sv, v)|\,ds \,e^{\a |v|^2} \,dxdv\\
=& \int_{\O \times \R^3} \left(\int_0^{\tau(y, -u)} e^{- \nu_0 t}\,dt\right) |f(y, u)| \,e^{\a |u|^2}\,dydu\\
\leq& \frac{1}{\nu_0} \int_{\O \times \R^3} |f(y, u)| \,e^{\a |u|^2}\,dydu.
\end{align*}
Here, we used Lemma \ref{lem:est_S_convex} in the last inequality. The above estimate implies that $\opnorm{S_\Omega}_{1, \alpha} \leq \nu_0^{-1}$.

For $1 < p < \infty$, we apply the H\"older inequality to have
\begin{align*}
&\int_{\O \times \R^3} \left| \int_0^{\tau(x, v)} e^{- \nu(v) s} f(x - sv, v)\,ds \right|^p \,e^{p \a |v|^2} \,dxdv\\
\leq& \int_{\O \times \R^3} \left( \int_0^{\tau(x, v)} e^{- \nu_0 s}\,ds \right)^{\frac{p}{p'}} \left( \int_0^{\tau(x, v)} e^{- \nu_0 s} |f(x - sv, v)|^p\,ds \right) \,e^{p \a |v|^2} \,dxdv\\
\leq& \frac{1}{\nu_0^{p - 1}} \int_{\O \times \R^3} \int_0^{\tau(x, v)} e^{- \nu_0 s} |f(x - sv, v)|^p \,e^{p \a |v|^2}\,ds \,dxdv\\
\leq& \frac{1}{\nu_0^p} \int_{\O \times \R^3} |f(y, u)|^p \,e^{p \a |u|^2}\,dydu,
\end{align*}
where $p' = p/(p-1)$. Thus, we have
\[
\| S_\O f \|_{L^p_\a(\O \times \R^3)}^p \leq \frac{1}{\nu_0^p} \| f \|_{L^p_\a(\O \times \R^3)}^p,
\]
which implies 
\[
\| S_\O f \|_{L^p_\a(\O \times \R^3)} \leq \nu_0^{-1} \| f \|_{L^p_\a(\O \times \R^3)}.
\]
This completes the proof.
\end{proof}

Next, we explore the boundedness of operators $J$ and $S$ acting on weighted spaces of bounded continuous functions. 

In order to study the continuity, we need to extend the definition of $Jg$. We let $\a \geq 0$ and introduce the function space 
\[
C_\alpha(\Gamma^-) := \{ g \in C(\Gamma^-) \mid \| g \|_{C_\alpha(\Gamma^-)} < \infty \},
\]
where
\[
\| g \|_{C_\alpha(\Gamma^-)} := \sup_{(x, v) \in \Gamma^-} |g(x, v)| e^{\alpha |v|^2}.
\]Recall that
\[
Jg(x, v) = e^{-\nu(v) \tau(x, v)} g(q(x, v), v),
\]
where $q(x, v) = x - \tau(x, v) v$. Since $\tau$ and $q$ are not defined 
at $v = 0$, $Jg$ is not defined there. We claim that it is continuous at $v = 0$ by defining $Jg(x, 0) = 0$ for $x \in \O$. Indeed, let $x \in \Omega$ and $d_x = \dist(x,\partial \O)$. Then, we have
\begin{equation*}
|Jg(x,v)| \leq \Vert g \Vert_{C_\alpha(\Gamma^-)} e^{- \frac{\nu_0 d_x}{|v|}} e^{-\a |v|^2}, 
\end{equation*} 
where the right hand side tends to 0 as $|v|$ tends to 0. Thus, $Jg$ is continuous at $v = 0$.  Since $\nu$, $\tau$ and $q$ are continuous on $(\O \times (\R^3 \setminus \{0\})) \cup \Gamma^\pm$, $Jg$ is also continuous. Hence, it is continuous on $(\Omega \times \R^3) \cup \Gamma^\pm$. 

We let $\alpha \geq 0$ and define the function space $C_\alpha((\O \times \R^3) \cup \Gamma^\pm)$ by
\[
C_\alpha((\O \times \R^3) \cup \Gamma^\pm) := \left\{ f \in C((\O \times \R^3) \cup \Gamma^\pm) \mid \| f \|_{C_\alpha((\O \times \R^3) \cup \Gamma^\pm)} < \infty \right\},
\]
where
\[
\| f \|_{C_\alpha((\O \times \R^3) \cup \Gamma^\pm)} := \sup_{(x, v) \in (\O \times \R^3) \cup \Gamma^\pm} |f(x, v)| e^{\alpha |v|^2}.
\]
Here we remark that, in general, a solution to the problem \eqref{integral form} is discontinuous on the grazing set $\Gamma_0$, where
\[
\Gamma_0 := \{ (x, v) \in \partial \Omega \times \R^3 \mid n(x) \cdot v = 0 \}.
\]
This is the main reason that we introduced the function space $C_\alpha((\O \times \R^3) \cup \Gamma^\pm)$ instead of $C_\alpha(\overline{\O} \times \R^3)$.

Concerning the boundedness of the operator $J$ and $S_\O$ on $C_\a$, we have the following observations.

\begin{proposition} \label{prop:bound_J_Calpha}
Let $\O$ be a bounded convex domain. Let $J$ be the operator as defined by \eqref{J}. Then, for all $g \in C_\alpha(\Gamma^-)$, we have $Jg \in C_\alpha((\O \times \R^3) \cup \Gamma^\pm)$. Moreover, we have
\[
\opnorm{J}_{\alpha} := \sup_{\substack{g \in C_\alpha(\Gamma^-) \\ g \neq 0}} \frac{\| Jg \|_{C_\alpha((\O \times \R^3) \cup \Gamma^\pm)}}{\| g \|_{C_\alpha(\Gamma^-)}} = 1.
\]
\end{proposition}

\begin{proof}
%We notice that the functions $\tau(x, v)$ and $q(x,v)$ are continuous on $\O \times (\R^3 \setminus \{0\})\cup \Gamma^\pm$ and $\nu(v)$ is continuous on $\R^3$. Therefore, $Jg$ is continuous on $\O \times (\R^3 \setminus \{0\})\cup \Gamma^\pm$ provided $g \in C_\alpha(\Gamma^-)$. 
The estimate of $\opnorm{J}_{\alpha}$ follows from the fact that the exponential factor of $Jg$ is less than or equal to $1$, which implies that the maximum is attained on the boundary $\Gamma^-$.
\end{proof}

%We give a boundedness estimate for the operator $S_\O$.

\begin{proposition} \label{prop:bound_S_Calpha}
Let $\O$ be a bounded convex domain. Let $S_\Omega$ be the operator as defined by \eqref{S}. Then, for all $f \in C_\alpha((\O \times \R^3) \cup \Gamma^\pm)$, we have $S_\O f \in C_\alpha((\O \times \R^3) \cup \Gamma^\pm)$. Moreover, we have
\[
\opnorm{S_\Omega}_{\alpha} := \sup_{\substack{f \in C_\alpha((\O \times \R^3) \cup \Gamma^\pm) \\ f \neq 0}} \frac{\| S_\O f \|_{C_\alpha((\O \times \R^3) \cup \Gamma^\pm)}}{\| f \|_{C_\alpha((\O \times \R^3) \cup \Gamma^\pm)}} \leq \frac{1}{\nu_0},
\]
where $\nu_0$ is the constant in \eqref{AA}.
\end{proposition}

\begin{proof}
Continuity of the function $S_\O f$ is obvious if $f \in C_\alpha((\O \times \R^3) \cup \Gamma^\pm)$. We prove that the function $S_\O f$ is bounded with the weight $e^{\alpha |v|^2}$. For $f \in C_\alpha((\O \times \R^3) \cup \Gamma^\pm)$, we have
\begin{align*}
e^{\alpha |v|^2} |S_\O f(x, v)| \leq& \int_0^{\tau(x, v)} e^{- \nu(v) t} |f(x - tv, v)| e^{\alpha |v|^2} \,dt\\
\leq& \| f \|_{C_\alpha((\O \times \R^3) \cup \Gamma^\pm)} \int_0^{\tau(x, v)} e^{- \nu_0 t} \,dt\\
\leq& \frac{1}{\nu_0} \| f \|_{C_\alpha((\O \times \R^3) \cup \Gamma^\pm)},
\end{align*}
which implies that $\| S_\O f \|_{C_\alpha((\O \times \R^3) \cup \Gamma^\pm)} \leq \nu_0^{-1} \| f \|_{C_\alpha((\O \times \R^3) \cup \Gamma^\pm)}$. This completes the proof.
\end{proof}

We summarize useful estimates for the integral kernel $k$, which are modified from \cite{CLT}.

\begin{lemma} \label{lem:identity_bound_k}
Let $a \in \R$ and $0 < \rho < 1$. Then, we have
\begin{equation} \label{eq:identity_bound_k}
\begin{split}
&-\frac{1 - \rho}{4} \left( |v - v^*|^2 + \left( \frac{|v|^2 - |v^*|^2}{|v - v^*|} \right)^2 \right)\\ 
=& a |v|^2 - \alpha_{1, a, \rho} |v - v^*|^2 -(1-\rho) \left( \frac{(v - v^*) \cdot v}{|v - v^*|} - \alpha_{2, a, \rho} |v - v^*| \right)^2 - a |v^*|^2\\
=& -a |v|^2 - \alpha_{1, a, \rho} |v - v^*|^2 -(1-\rho) \left( \frac{(v - v^*) \cdot v^*}{|v - v^*|} + \alpha_{2, a, \rho} |v - v^*| \right)^2 + a |v^*|^2
\end{split}
\end{equation}
for all $v, v^* \in \R^3$, where
\begin{align} 
\alpha_{1, a, \rho} :=& \frac{(1 - \rho + 2a)(1 - \rho - 2a)}{4(1 -\rho)}, \label{alpha_1}\\
\alpha_{2, a, \rho} :=& \frac{1 - \rho - 2a}{2(1-\rho)}. \label{alpha_2}
\end{align}
\end{lemma}

\begin{proof}
We start from the following identity:
\[
|v^*|^2 = |v - v^*|^2 - 2 (v - v^*) \cdot v + |v|^2.
\]
Then, we have
\[
a (|v|^2 + |v - v^*|^2 - 2 (v - v^*) \cdot v - |v^*|^2) = 0
\]
for all $a \in \R$ and
\[
\frac{|v|^2 - |v^*|^2}{|v - v^*|} = 2 \frac{(v - v^*) \cdot v}{|v - v^*|} - |v - v^*|.
\]
Hence, we have
\begin{align*}
&-\frac{1-\rho}{4} \left( |v - v^*|^2 + \left( \frac{|v|^2 - |v^*|^2}{|v - v^*|} \right)^2 \right) \\
=& a |v|^2 -\frac{1}{2}(1-\rho-2a) |v - v^*|^2 -(1-\rho)\left(\frac{(v - v^*) \cdot v}{|v - v^*|} \right)^2\\ 
&+ (1-\rho-2a) (v - v^*) \cdot v - a |v^*|^2\\
=& a |v|^2 - \alpha_{1, a, \rho} |v - v^*|^2 -(1-\rho) \left( \frac{(v - v^*) \cdot v}{|v - v^*|} - \alpha_{2, a, \rho} |v - v^*| \right)^2 - a |v^*|^2.
\end{align*}
Therefore, the first identity in \eqref{eq:identity_bound_k} is proved. The second identity is proved in the same way with the identity:
\[
|v|^2 = |v - v^*|^2 + 2 (v - v^*) \cdot v^* + |v^*|^2.
\]
\end{proof}

\begin{corollary} \label{cor:est_bound_k}
Let $a \in \R$ and $0 < \rho < 1$. Then we have
\[
|k(v, v^*)| \lesssim |v - v^*|^{-1} e^{-a |v|^2} e^{-\a_{1, a, \rho} |v - v^*|^2} e^{a |v^*|^2}, 
\]
where $\a_{1, a, \rho}$ is the constant defined by \eqref{alpha_1}.
\end{corollary}

\begin{proof}
The estimate follows from the identity \eqref{eq:identity_bound_k} and the estimate
\[
e^{-(1 - \rho) \left( \frac{(v - v^*) \cdot v}{|v - v^*|} - \a_{2, a, \rho} |v - v^*| \right)^2} \leq 1.
\]
\end{proof}

Based on the identity \eqref{eq:identity_bound_k} and the argument in \cite{Caf 1}, we have the following estimate.

\begin{lemma} \label{lem:est_E_alpha}
Let $\mu_1$ and $\mu_2$ be two real numbers such that $0 \leq \mu_1 < 3$ and $\mu_2 > 0$. Also, let $0 < \rho < 1$. Then, for any $- \mu_2 (1 - \rho)/2 < a < \mu_2 (1 - \rho)/2$, we have
\[
\int_{\R^3} \frac{1}{|v - v^*|^{\mu_1}} E_\rho(v, v^*)^{\mu_2} e^{a |v^*|^2}\,dv^* \lesssim (1 + |v|)^{-1} e^{a |v|^2},
\]
where $E_\rho$ is the function defined by \eqref{E}.
\end{lemma}

\begin{proof}
By Lemma \ref{lem:identity_bound_k}, we have
\begin{align*}
&\frac{1}{|v - v^*|^{\mu_1}} E_\rho(v, v^*)^{\mu_2} e^{a |v^*|^2}\\
=& \frac{1}{|v - v^*|^{\mu_1}} e^{a|v|^2} e^{- \alpha_{1, a, \rho}' |v - v^*|^2 - \mu_2 (1 - \rho) \left( \frac{(v - v^*) \cdot v}{|v - v^*|} - \alpha_{2, a, \rho}' |v - v^*| \right)^2},
\end{align*}
where
\begin{align*}
\a_{1, a, \rho}' :=& \mu_2 \a_{1, a/\mu_2, \rho} = \frac{(\mu_2(1 - \rho) + 2a)(\mu_2(1 - \rho) - 2a)}{4 \mu_2 (1 - \rho)},\\
\a_{2, a, \rho}' :=& \alpha_{2, a/\mu_2, \rho} = \frac{\mu_2(1 - \rho) - 2a}{2\mu_2(1 - \rho)}.
\end{align*}
Notice that the constant $\a_{1, a, \rho}'$ is positive as long as $- \mu_2 (1 - \rho)/2 < a < \mu_2 (1 - \rho)/2$. Thus, we have
\begin{align*}
&\int_{\R^3} \frac{1}{|v - v^*|^{\mu_1}} E_\rho(v, v^*)^{\mu_2} e^{a |v^*|^2}\,dv^*\\
=& e^{a|v|^2} \int_{\R^3} \frac{1}{|v - v^*|^{\mu_1}} e^{- \alpha_{1, a, \rho}' |v - v^*|^2 - \mu_2 (1 - \rho) \left( \frac{(v - v^*) \cdot v}{|v - v^*|} - \alpha_{2, a, \rho}' |v - v^*| \right)^2}\,dv^*.
\end{align*}

Now we shall prove that
\[
\int_{\R^3} \frac{1}{|v - v^*|^{\mu_1}} e^{- \alpha_{1, a, \rho}' |v - v^*|^2 - \mu_2 (1 - \rho) \left( \frac{(v - v^*) \cdot v}{|v - v^*|} - \alpha_{2, a, \rho}' |v - v^*| \right)^2}\,dv^* \lesssim (1 + |v|)^{-1}.
\]
We introduce the spherical coordinates: $v^* = v + r \omega$ for $r > 0$ and $\omega \in S^2$. Since the integrand depends only on $r$ and $t := \omega \cdot v/|v|$, we have,
\[
\begin{split}
&\int_{\R^3} \frac{1}{|v - v^*|^{\mu_1}} e^{- \alpha_{1, a, \rho}' |v - v^*|^2 - \mu_2 (1 - \rho) \left( \frac{(v - v^*) \cdot v}{|v - v^*|} - \alpha_{2, a, \rho}' |v - v^*| \right)^2}\,dv^*\\ 
=& 2 \pi \int_0^\infty \int_{-1}^1 r^{2 - \mu_1} e^{- \alpha_{1, a, \rho}' r^2 - \mu_2 (1 - \rho) \left( |v| t - \alpha_{2, a, \rho}' r \right)^2}\,dtdr.
\end{split}
\]

We estimate the $t$-integral. When $0 < |v| \leq 1$, we have
\[
\int_{-1}^1 e^{- \mu_2 (1 - \rho) \left( |v| t - \alpha_{2, a, \rho}' r \right)^2}\,dt \leq 2.
\]
On the other hand, for $|v| > 1$, let $s := |v|t - \a_{2, a, \rho}'r$. Then, we have
\[
\begin{split}
\int_{-1}^1 e^{- \mu_2 (1 - \rho) \left( |v| t - \alpha_{2, a, \rho}' r \right)^2}\,dt =& \frac{1}{|v|} \int_{-|v| - \a_{2, a, \rho}'}^{|v| - \a_{2, a, \rho}'} e^{- \mu_2 (1 - \rho) s^2}\,ds\\
\leq& \frac{1}{|v|} \int_{-\infty}^{\infty} e^{- \mu_2 (1 - \rho) s^2}\,ds\\
\lesssim& \frac{1}{|v|}.
\end{split}
\]
Thus, we have
\[
\int_{-1}^1 e^{- \mu_2 (1 - \rho) \left( |v| t - \alpha_{2, a, \rho}' r \right)^2}\,dt \lesssim (1 + |v|)^{-1}.
\]
Since $0 \leq \mu_1 < 3$, we have
\begin{align*}
&\int_{\R^3} \frac{1}{|v - v^*|^{\mu_1}} e^{- \alpha_{1, a, \rho}' |v - v^*|^2 - \mu_2 (1 - \rho) \left( \frac{(v - v^*) \cdot v}{|v - v^*|} - \alpha_{2, a, \rho}' |v - v^*| \right)^2}\,dv^*\\ 
\lesssim& (1 + |v|)^{-1} \int_0^\infty r^{2 - \mu_1} e^{- \alpha_{1, a, \rho}' r^2}\,dr\\
\lesssim& (1 + |v|)^{-1}.
\end{align*}
This completes the proof.
\end{proof}

Due to the assumptions \eqref{AB'} and \eqref{AC'}, we have the following estimates.

\begin{corollary} \label{cor:est_K_alpha}
Let $0 < \mu < 3$ and $0 < \rho < 1$. Then, for any $-\mu(1 - \rho)/2 < a < \mu(1 - \rho)/2$, we have
\[
\int_{\R^3} |k(v, v^*)|^\mu e^{a |v^*|^2}\,dv^* \lesssim (1 + |v|)^{-1} e^{a |v|^2}.
\]
\end{corollary}

\begin{corollary} \label{cor:est_Kdv_alpha}
Let $0 < \mu < 3/2$ and $0 < \rho < 1$. Then, for any $-\mu(1 - \rho)/2 < a < \mu(1 - \rho)/2$, we have
\[
\int_{\R^3} |\nabla_v k(v, v^*)|^\mu e^{a |v^*|^2}\,dv^* \lesssim (1 + |v|)^{\mu - 1} e^{a |v|^2}.
\]
\end{corollary}

For the $v^*$ variable, we have the following estimate.

\begin{lemma} \label{lem:est_K_alpha*}
Let $\mu_1$, $\mu_2$ and $\mu_3$ be three real numbers such that $\mu_1 \geq 0$, $\mu_2 \geq 0$, $\mu_3 > 0$ and $\mu_1 + \mu_2 < 3$. Also, let $0 < \rho < 1$. Then, for any $- \mu_3 (1 - \rho)/2 < a < \mu_3 (1 - \rho)/2$, we have
\[
\int_{\R^3} \frac{1}{|v|^{\mu_1}} \frac{1}{|v - v^*|^{\mu_2}} E_\rho(v, v^*)^{\mu_3} e^{a |v|^2}\,dv \lesssim e^{a |v^*|^2}.
\]
\end{lemma}

\begin{proof}
By Lemma \ref{lem:identity_bound_k}, we have
\begin{align*}
&\frac{1}{|v|^{\mu_1}} \frac{1}{|v - v^*|^{\mu_2}} E_\rho(v, v^*)^{\mu_3} e^{a |v|^2}\\
=&\frac{1}{|v|^{\mu_1}} \frac{1}{|v - v^*|^{\mu_2}} e^{a|v^*|^2} e^{- \alpha_{1, a, \rho}'' |v - v^*|^2 - \mu_3 (1 - \rho) \left( \frac{(v - v^*) \cdot v}{|v - v^*|} + \alpha_{2, a, \rho}'' |v - v^*| \right)^2},
\end{align*}
where
\begin{align*}
\a_{1, a, \rho}'' :=& \mu_3 \a_{1, a/\mu_3, \rho} = \frac{(\mu_3(1 - \rho) + 2a)(\mu_3(1 - \rho) - 2a)}{4\mu_3(1 - \rho)},\\
\a_{2, a, \rho}'' :=& \alpha_{2, a/\mu_3, \rho} = \frac{\mu_3(1 - \rho) - 2a}{2\mu_3(1 - \rho)}.
\end{align*}
Notice that the constant $\a_{1, a, \rho}''$ is positive as long as $- \mu_3 (1 - \rho)/2 < a < \mu_3 (1 - \rho)/2$. Thus, we have
\begin{align*}
&\int_{\R^3} \frac{1}{|v|^{\mu_1}} \frac{1}{|v - v^*|^{\mu_2}} E_\rho(v, v^*)^{\mu_3} e^{a |v|^2}\,dv\\
=& e^{a|v^*|^2} \int_{\R^3} \frac{1}{|v|^{\mu_1}} \frac{1}{|v - v^*|^{\mu_2}} e^{- \alpha_{1, a, \rho}'' |v - v^*|^2 - \mu_3 (1 - \rho) \left( \frac{(v - v^*) \cdot v}{|v - v^*|} + \alpha_{2, a, \rho}'' |v - v^*| \right)^2}\,dv^*\\
\leq& e^{a|v^*|^2} \int_{\{|v| \leq |v - v^*| \}} \frac{1}{|v|^{\mu_1}} \frac{1}{|v - v^*|^{\mu_2}} e^{- \alpha_{1, a, \rho}'' |v - v^*|^2}\,dv\\ 
&+ e^{a|v^*|^2} \int_{\{|v| \geq |v - v^*| \}} \frac{1}{|v|^{\mu_1}} \frac{1}{|v - v^*|^{\mu_2}} e^{- \alpha_{1, a, \rho}'' |v - v^*|^2}\,dv\\
\leq& e^{a|v^*|^2} \int_{\{|v| \leq |v - v^*| \}} \frac{1}{|v|^{\mu_1 + \mu_2}} e^{- \alpha_{1, a, \rho}'' |v|^2}\,dv\\ 
&+ e^{a|v^*|^2} \int_{\{|v| \geq |v - v^*| \}} \frac{1}{|v - v^*|^{\mu_1 + \mu_2}} e^{- \alpha_{1, a, \rho}'' |v - v^*|^2}\,dv\\
\lesssim& e^{a|v^*|^2}.
\end{align*}
This completes the proof.
\end{proof}

In the same fashion, we obtain the following estimates.

\begin{corollary} \label{cor:est_K_alpha*}
Let $\mu_1$ and $\mu_2$ be two real numbers such that $\mu_1 \geq 0$, $\mu_2 > 0$ and $\mu_1+ \mu_2 < 3$. Assume $0 < \rho < 1$. Then, for any $-\mu_2(1 - \rho)/2 < a < \mu_2(1 - \rho)/2$, we have
\[
\int_{\R^3} \frac{1}{|v|^{\mu_1}} |k(v, v^*)|^{\mu_2} e^{a |v|^2}\,dv \lesssim e^{a |v^*|^2}
\]
and
\[
\int_{\R^3} \frac{1}{|v^*|^{\mu_1}} |k(v, v^*)|^{\mu_2} e^{a |v^*|^2}\,dv^* \lesssim e^{a |v|^2}.
\]
\end{corollary}

\begin{corollary} \label{cor:est_Kdv_alpha*}
Let $\mu_1$ and $\mu_2$ be two real numbers such that $\mu_1 \geq \mu_2 > 0$ and $\mu_1+ 2\mu_2 < 3$. Also, let $0 < \rho < 1$. Then, for any $a$ satisfying $-\mu_2(1 - \rho)/2 < a < \mu_2(1 - \rho)/2$, we have
\[
\int_{\R^3} \frac{1}{|v|^{\mu_1}} |\nabla_v k(v, v^*)|^{\mu_2} e^{a |v|^2}\,dv \lesssim (1 + |v^*|)^{\mu_2 - 1} e^{a |v^*|^2}.
\]
\end{corollary}

Noticing that 
\[
\frac{1 + |v|}{|v - v^*|^2} \leq \frac{1}{|v - v^*|} + \frac{1 + |v^*|}{|v - v^*|^2},
\]
and switching variables $v$ and $v^*$ in Corollary \ref{cor:est_Kdv_alpha}, we obtain the following estimate.

\begin{corollary} \label{cor:est_Kdv_alpha*2}
Let $0 < \mu < 3/2$ and $0 < \rho < 1$. Then, for any $-\mu(1 - \rho)/2 < a < \mu(1 - \rho)/2$, we have
\[
\int_{\R^3} |\nabla_v k(v, v^*)|^{\mu} e^{a |v|^2}\,dv \lesssim (1 + |v^*|)^{-1} e^{a |v^*|^2}.
\]
\end{corollary}

With the aforementioned preparation, we are now in the position to show the existence and uniqueness of solutions to the integral equation \eqref{integral form} in $L^p_\a$ and $C_\a$ spaces.

\begin{proof}[Proof of Lemma \ref{estimate on Lp}]
When $p = 1$, applying Lemma \ref{change of variable lemma}, Lemma \ref{lem:est_S_convex} and Corollary \ref{cor:est_K_alpha*}, we have
\begin{align*}
&\int_{\O \times \mathbb{R}^{3}} | S_{\Omega}Kh(x,v) | e^{\a |v|^2}\,dxdv\\
=&\int_{\mathbb{R}^{3}}\int_{\Omega}\left| \int_{0}^{\tau(x, v)} e^{-\nu(v)s}Kh(x-sv,v)ds \right| e^{\a |v|^2}\,dxdv\\
\leq&\int_{\mathbb{R}^{3}}\int_{\Omega} \int_{0}^{\tau(x, v)} e^{-\nu_0 s}\left|Kh(x-sv,v)\right|\,ds \,e^{\a |v|^2}\,dxdv\\
\lesssim& \int_{\mathbb{R}^{3}}\int_{\Omega} \frac{\diam(\O)}{|u|} \int_{\R^3} |k(u, v^*)| |h(y, v^*)|\,dv^* \, e^{\a |u|^2}\,dydu\\
\lesssim& \diam(\O) \int_{\R^3} \int_\O |h(y, v^*)| e^{\a |v^*|^2}\,dydv^*.
\end{align*}
For $1 < p <\infty$ , by the H\"older inequality and Lemma \ref{lem:est_S_convex}, we have
\begin{align*}
&\int_{\O \times \mathbb{R}^{3}}\int_{\Omega}| S_{\Omega}Kh(x,v) |^p e^{p \a |v|^2}\,dxdv\\
=&\int_{\mathbb{R}^{3}}\int_{\Omega}\left| \int_{0}^{\tau(x, v)} e^{-\nu(v)s}Kh(x-sv,v)ds \right|^p e^{p \a |v|^2}\,dxdv\\
\leq&\int_{\mathbb{R}^{3}}\int_{\Omega}\left( \int_{0}^{\tau(x, v)} e^{-\nu_0 s}\,ds \right)^{\frac{p}{p'}} \left( \int_{0}^{\tau(x, v)} e^{-\nu_0 s}\left|Kh(x-sv,v)\right|^p\,ds \right) e^{p \a |v|^2}\,dxdv\\
\lesssim& \int_{\mathbb{R}^{3}}\int_{\Omega} \int_{0}^{\tau(x, v)} e^{-\nu_0 s}\left|Kh(x-sv,v)\right|^p \,ds \, e^{p \a |v|^2}\,dxdv,
\end{align*}
where $p'$ is the H\"older conjugate of $p$. 

In the same way as for the case $p = 1$, we have
\begin{align*}
&\int_{\mathbb{R}^{3}}\int_{\Omega} \int_{0}^{\tau(x, v)} e^{-\nu_0 s}\left|Kh(x-sv,v)\right|^p \,ds \, e^{p \a |v|^2}\,dxdv\\
\lesssim &\int_{\mathbb{R}^{3}} \int_{\Omega} \frac{\diam(\O)}{|u|} \left| \int_{\R^3} k(u, v^*) h(y, v^*)\,dv^* \right|^p \,e^{p \a |u|^2}\,dy du.
\end{align*}

Take $\a_1$ satisfying
\begin{equation} \label{alpha1_cond1}
-\frac{1 - \rho}{2} < - p' \a_1 < \frac{1 - \rho}{2}
\end{equation}
and
\begin{equation} \label{alpha1_cond2}
-\frac{1 - \rho}{2} < p (\a - \a_1) < \frac{1 - \rho}{2}.
\end{equation}
We check that we can indeed take such a constant $\a_1$. By solving inequalities \eqref{alpha1_cond1} and \eqref{alpha1_cond2} in terms of $\a_1$, we have
\begin{equation} \label{alpha1_cond1'}
-\frac{1 - \rho}{2p'} < \a_1 < \frac{1 - \rho}{2p'}
\end{equation}
and
\begin{equation} \label{alpha1_cond2'}
\a -\frac{1 - \rho}{2p} < \a_1 < \a + \frac{1 - \rho}{2p}.
\end{equation}
Since $\a \geq 0$, we have $-(1 - \rho)/2p' < 0 < \a + (1 - \rho)/2p$. So, it suffices to show that $\a - (1 - \rho)/2p < (1 - \rho)/2p'$. It is equivalent to 
\[
\a < \frac{1 - \rho}{2p} + \frac{1 - \rho}{2p'} = \frac{1 - \rho}{2},
\]
which is the assumption of Lemma \ref{estimate on Lp}. Thus, we can indeed take a constant $\a_1$ satisfying \eqref{alpha1_cond1} and \eqref{alpha1_cond2}.

Then, by Corollary \ref{cor:est_K_alpha} and Corollary \ref{cor:est_K_alpha*}, we have
\begin{align*}
&\int_{\mathbb{R}^{3}}\int_{\Omega}\frac{1}{|u|} \left| \int_{\mathbb{R}^{3}}k(u,v^{*})h(y,v^{*}) \,dv^{*} \right|^p \,e^{p \a |u|^2}dydu\\
\leq&\int_{\mathbb{R}^{3}}\int_{\Omega}\frac{1}{|u|} \left( \int_{\mathbb{R}^{3}} | k(u,v^{*}) | e^{-p' \a_1 |v^*|^2} \,dv^{*} \right)^{\frac{p}{p'}}\\
&\times \left( \int_{R^{3}} |k(u,v^{*})|  |h(y,v^{*})|^p e^{p \a_1 |v^*|^2} \,dv^{*} \right) e^{p \a |u|^2}dydu\\
\lesssim&\int_{\mathbb{R}^{3}}\int_{\Omega}\frac{1}{|u|} \left( \int_{\mathbb{R}^{3}} | k(u,v^{*}) | |h(y,v^{*})|^p e^{p \a_1 |v^*|^2} \,dv^{*} \right)\,e^{p (\a - \a_1) |u|^2} dydu\\
=&\int_{\mathbb{R}^{3}}\int_{\Omega} \left( \int_{\mathbb{R}^{3}} \frac{1}{|u|}  | k(u,v^{*}) | e^{p (\a - \a_1) |u|^2} \,du \right) |h(y,v^{*})|^p\,e^{p \a_1 |v^*|^2}dydv^{*}\\
\lesssim& \int_{\mathbb{R}^{3}} e^{p (\a - \a_1) |v^*|^2} \int_{\Omega} |h(y,v^{*})|^p \,e^{p \a_1 |v^*|^2}dydv^{*}\\
=& \int_{\mathbb{R}^{3}} \int_{\Omega} |h(y,v^{*})|^p \,e^{p \a |v^*|^2}dydv^{*}.
\end{align*}

Summarizing the above estimates, we obtain
\[
\int_{\O \times \mathbb{R}^{3}} | S_{\Omega}Kh(x,v) |^p\,e^{p \a |v|^2}\,dxdv \lesssim \diam(\Omega) \int_{\O \times \mathbb{R}^{3}} |h(x,v)|^p\,e^{p \a |v|^2}\,dxdv.
\]
This completes the proof.
\end{proof}

By Lemma \ref{estimate on Lp}, we obtain
\[
\opnorm{S_\Omega K}_{p, \alpha} \lesssim \diam(\O)^{\frac{1}{p}}
\]
for all $1 \leq p < \infty$ and $0 \leq \a < (1 - \rho)/2$. In particular, for any $1 \leq p < \infty$ and $0 \leq \a < (1 - \rho)/2$, there is a constant $\epsilon(p, \a) > 0$ such that $\opnorm{S_\O K}_{p, \alpha} < 1$ for all $0 < \diam(\O) < \epsilon(p, \a)$. In that case, from the contraction mapping theorem, for any given $Jg \in L^p_\alpha(\O \times \R^3)$, the equation \eqref{integral form} has the unique solution $f \in L^p_\alpha(\O \times \R^3)$.

%the number $1$ belongs to the resolvent of the operator $S_\O K$ acting on $L^p_\alpha(\O \times \R^3)$. 

In order to show well-posedness of the integral equation \eqref{integral form} in $C_\a((\O \times \R^3) \cup \Gamma^\pm)$, we iterate the integral to obtain
\begin{equation} \label{integral form2}
f = Jg + S_\O K Jg + S_\O K S_\O K f.
\end{equation}
We discuss the existence and uniqueness of the solution to the equation \eqref{integral form2} in the weighted space $C_\alpha((\O \times \R^3) \cup \Gamma^\pm)$ instead of \eqref{integral form}. We remark that $|f(x, v)| \leq \| f \|_{C_\alpha((\O \times \R^3) \cup \Gamma^\pm)} e^{- \alpha |v|^2}$ for all $(x, v) \in (\O \times \R^3) \cup \Gamma^\pm$ and for any $f \in C_\alpha((\O \times \R^3) \cup \Gamma^\pm)$.

\begin{lemma} \label{lem:small_second_convex}
Let $0 \leq \alpha < (1 - \rho)/2$, where $\rho$ is a constant in {\bf Assumption A}. Then, we have
$$
\| K S_\O K h \|_{C_\alpha((\O \times \R^3) \cup \Gamma^\pm)} \lesssim \diam(\O) \| h \|_{C_\alpha((\O \times \R^3) \cup \Gamma^\pm)}
$$
for all $h \in C_\alpha((\O \times \R^3) \cup \Gamma^\pm)$.
\end{lemma}

\begin{proof}
By Corollary \ref{cor:est_bound_k}, we have
\[
|Kh(x, v)| \lesssim e^{-\alpha|v|^2} \| h \|_{C_\alpha((\O \times \R^3) \cup \Gamma^\pm)}
\]
for all $(x, v) \in (\O \times \R^3) \cup \Gamma^\pm$ and for all $h \in C_\alpha((\O \times \R^3) \cup \Gamma^\pm)$. Thus, by Lemma \ref{lem:est_S_convex} and Corollary \ref{cor:est_K_alpha*}, we have
\begin{align*}
&|(K S_\O K h)(x, v)|\\ 
\lesssim& \| h \|_{C_\alpha((\O \times \R^3) \cup \Gamma^\pm)} \int_{\R^3} |k(v, v^*)| \int_0^{\tau(x, v^*)} e^{-\nu(|v^*|) t} e^{-\alpha |v^*|^2}\,dt\,dv^*\\ 
\lesssim& \| h \|_{C_\alpha((\O \times \R^3) \cup \Gamma^\pm)} \int_{\R^3} |k(v, v^*)| \min \left\{ \frac{1}{\nu_0}, \frac{\diam(\O)}{|v^*|}\right\} e^{-\alpha |v^*|^2}\,dv^*\\
\lesssim& \diam(\O) \| h \|_{C_\alpha((\O \times \R^3) \cup \Gamma^\pm)} \int_{\R^3} \frac{|k(v, v^*)|}{|v^*|} e^{-\alpha |v^*|^2}\,dv^*\\
\lesssim& \diam(\O) \| h \|_{C_\alpha((\O \times \R^3) \cup \Gamma^\pm)} e^{-\alpha |v|^2}
\end{align*}
for all $(x, v) \in (\O \times \R^3) \cup \Gamma^\pm$, which implies that 
\[
\| K S_\O K h \|_{C_\alpha((\O \times \R^3) \cup \Gamma^\pm)} \lesssim \diam(\O) \| h \|_{C_\alpha((\O \times \R^3) \cup \Gamma^\pm)}.
\]
This completes the proof.
\end{proof}

By Proposition \ref{prop:bound_S_Calpha} and Lemma \ref{lem:small_second_convex}, we have
\[
\opnorm{S_\O K S_\O K}_\alpha \lesssim \diam(\O).
\]
In particular, there is a constant $\epsilon > 0$ such that $\opnorm{S_\O K S_\O K}_\alpha < 1$ if $\diam(\O) < \epsilon$. In that case, from the contraction mapping theorem again, for any given $Jg \in C_\alpha((\O \times \R^3) \cup \Gamma^\pm)$, the equation \eqref{integral form2} has the unique solution $f \in C_\alpha((\O \times \R^3) \cup \Gamma^\pm)$. Moreover, we have
\begin{equation} \label{eq:uniform_est}
\| f \|_{C_\a((\O \times \R^3) \cup \Gamma^\pm)} \lesssim \| g \|_{C_\a(\Gamma^-)}
\end{equation}
for sufficiently small $\epsilon > 0$.

%%%%%%%%%%%%%%%%%%%%%%%%%%%%%%%%%%%%%%%%%%%%%%%%%%%%%%%%%%%%%%%%%%%%%%%%%%%%%%%%%%%%%%%%%%%%%%%%%%%%%%%
\section{Regularity on small convex domains} \label{sec:bounded_domain}
%%%%%%%%%%%%%%%%%%%%%%%%%%%%%%%%%%%%%%%%%%%%%%%%%%%%%%%%%%%%%%%%%%%%%%%%%%%%%%%%%%%%%%%%%%%%%%%%%%%%%%%

In this section, we discuss the existence and regularity of the solution to the integral equation \eqref{integral form}. In particular, we study convergence of the series \eqref{Picard} in $W^{1, p}_\a$ spaces.

To this aim, we give $L^p_\a$ estimates for $x$ and $v$ derivatives for the function $S_\O K h$ with $h \in W^{1, p}_\a(\O \times \R^3)$ for the case $1 \leq p < 2$. 

For the $x$ derivative, we have the following estimate.

\begin{lemma} \label{estimate on W1px for 1 to 2}
Let $1 \leq p < 2$ and $0 \leq \alpha < (1 - \rho)/2$, where $\rho$ is the constant in {\bf Assumption A}. Then, for $h \in W^{1, p}_\a(\O \times \R^3)$, we have
\[
\| \nabla_x S_\O K h \|_{L^p_\a(\O \times \R^3)} \lesssim \diam(\O)^\frac{1}{p} \| h \|_{W^{1, p}_\a(\O \times \R^3)} + \diam(\O)^{\frac{1}{p}} \| h \|_{L^p_\a(\p \O \times \R^3)}.
\]
\end{lemma}

\begin{proof}
Observe that 
\[
\nabla_{x}S_{\Omega}Kh = S_{\Omega}K \nabla_{x} h + S_{\O, x} Kh,
\]
where
\[
S_{\O, x} h(x, v) := (\nabla_{x}\tau(x, v)) e^{-\nu(v)\tau(x, v)} h(q(x,v),v).
\]
By Lemma \ref{estimate on Lp}, we have
\[
\int_{\O \times \mathbb{R}^{3}} | S_{\Omega}K\nabla_{x}h(x,v) |^p e^{p \a |v|^2}\,dxdv \lesssim \diam(\Omega) \int_{\O \times \mathbb{R}^{3}} | \nabla_{x}h(x,v) |^p e^{p \a |v|^2} \,dxdv.
\]
Thus, we focus on the estimate for the second term. In other words, we prove that 
\[
\| S_{\O, x} K h \|_{L^p_\a(\O \times \R^3)} \lesssim \diam(\O)^{\frac{1}{p}} \| h \|_{L^p_\a(\p \O \times \R^3)}.
\]

It is known in \cite{GuoKim} that
\begin{equation} \label{tau_dx}
\nabla_{x}\tau(x, v)=\frac{-n(q(x,v))}{N_-(x,v)|v|},
\end{equation}
where
\[
N_-(x, v) := N(q(x, v), v), \quad (x, v) \in \O \times \R^3.
\]
Now we perform the change of variables $z = q(x, v)$ and $s = \tau(x, v)$. Noting that $z \in \Gamma^-_v$, we have 
\begin{equation} 
\begin{split} \label{est:Lp_domain_to_boundary}
&\int_{\O \times \mathbb{R}^{3}} | (\nabla_x \tau(x, v)) e^{-\nu(v)\tau(x, v)}Kh(q(x,v),v) |^p e^{p \a |v|^2} \,dxdv\\
\leq& \int_{\mathbb{R}^{3}}\int_{\Omega} \frac{1}{N_-(x,v)^p |v|^p} e^{-p \nu_0 \tau(x, v)}| Kh(q(x,v),v) |^p e^{p \a |v|^2} \,dxdv\\
=& \int_{\mathbb{R}^{3}}\int_{\Gamma^-_v}\int_{0}^{\tau(z, -v)} \frac{1}{N(z,v)^p |v|^p}e^{-p \nu_0 s}| Kh(z,v) |^p N(z,v)|v| e^{p \a |v|^2}\,dsd\Sigma(z)dv\\
\lesssim& \diam(\O) \int_{\R^3} \int_{\Gamma^-_v} \frac{1}{|v|^p} \frac{1}{N(z,v)^{p - 1}}| Kh(z,v) |^p \, e^{p \a |v|^2} d\Sigma(z) dv\\
=& \diam(\O) \int_{\partial \O} \int_{\Gamma^-_z} \frac{1}{|v|^p} \frac{1}{N(z,v)^{p - 1}}| Kh(z,v) |^p \, e^{p \a |v|^2} dv d\Sigma(z).
\end{split}
\end{equation}
Here, we have used Corollary \ref{cor:est_S_convex}.

In the case $p = 1$, we have
\begin{align*}
&\int_{\partial \O} \int_{\Gamma^-_z} \frac{1}{|v|^p} \frac{1}{N(z,v)^{p - 1}}| Kh(z,v) |^p \, e^{p \a |v|^2} dv d\Sigma(z)\\
\leq& \int_{\mathbb{R}^{3}} \int_{\partial\Omega} \left( \int_{\R^3} \frac{1}{|v|} |k(v, v^*)| e^{\a |v|^2} \,dv \right) |h(z, v^*)| \, d\Sigma(z)dv^*\\
\lesssim& \int_{\partial \Omega \times \mathbb{R}^{3}} |h(z, v^*)| e^{\a |v^*|^2}\, d\Sigma(z)dv^*.
\end{align*}

For $1 < p < 2$ , we fix the variable $z$ and decompose the velocity $v$ into two components: $v = v_n + v_t$, where $v_n := (v \cdot n(z)) n(z)$ and $v_t := v - v_n$. Then, we have
\begin{align*}
&\int_{\Gamma^-_z} \frac{1}{|v|^p} \frac{1}{N(z,v)^{p - 1}}| Kh(z,v) |^p \, e^{p \a |v|^2} dv\\
\lesssim& \int_{\Gamma^-_z} \frac{1}{|v|^p} \frac{1}{N(z,v)^{p - 1}} \left( \int_{\R^3} |k(v, v^*)| |h(z, v^*)|^p e^{p \a_1 |v^*|}\,dv^* \right) \, e^{p (\a - \a_1) |v|^2} dv\\
=& \int_{\R^3} \left( \int_{\Gamma^-_z} \frac{1}{|v|} \frac{1}{|v_n|^{p - 1}} |k(v, v^*)| e^{p (\a - \a_1) |v|^2}\,dv \right) |h(z, v^*)|^p\,e^{p \a_1 |v^*|^2} dv^*\\
\lesssim& \int_{\R^3} \left( \int_{\Gamma^-_z} \frac{1}{|v|} \frac{1}{|v_n|^{p - 1}} |v_t - v^*_t|^{-1} e^{- \a_{1, p(\a - \a_1), \rho} |v_t - v^*_t|^2}\,dv \right) |h(z, v^*)|^p\,e^{p \a |v^*|^2} dv^*,
\end{align*}
where $\a_1$ is the constant satisfying \eqref{alpha1_cond1} and \eqref{alpha1_cond2}, and $\a_{1, p(\a - \a_1), \rho}$ is the constant in \eqref{alpha_1} with $a = p(\a - \a_1)$. By the rotation, we simply denote $v_n = r (1, 0, 0)$ for $r > 0$ and $v_t = (0, v_b)$ for $v_b \in \R^2$. Then, we further have
\begin{align*}
&\int_{\Gamma^-_z} \frac{1}{|v|} \frac{1}{|v_n|^{p - 1}} |v_t - v^*_t|^{-1} e^{- \a_{1, p(\a - \a_1), \rho} |v_t - v^*_t|^2}\,dv\\
=& \int_{\R^2} \left( \int_0^\infty \frac{1}{\sqrt{r^2 + |v_b|^2}} \frac{1}{r^{p-1}}\,dr \right) |v_b - v^*_b|^{-1} e^{- \a_{1, p(\a - \a_1), \rho} |v_b - v^*_b|^2}\,dv_b\\
=& \int_{\R^2} \left( \int_0^\infty \frac{1}{\sqrt{r^2 + 1}} \frac{1}{r^{p-1}}\,dr \right) |v_b|^{1 - p} |v_b - v^*_b|^{-1} e^{- \a_{1, p(\a - \a_1), \rho} |v_b - v^*_b|^2}\,dv_b.
\end{align*}
Since $0 < p - 1 < 1$, we have
\[
\int_0^\infty \frac{1}{\sqrt{r^2 + 1}} \frac{1}{r^{p-1}}\,dr \lesssim 1.
\]
Also, following the proof of Lemma \ref{lem:est_K_alpha*}, we have
\[
\int_{\R^2} |v_b|^{1 - p} |v_b - v^*_b|^{-1} e^{- \a_{1, p(\a - \a_1), \rho} |v_b - v^*_b|^2}\,dv_b \lesssim 1.
\]
Therefore, we have
\begin{align*}
\int_{\Gamma^-_z} \frac{1}{|v|^p} \frac{1}{N(z,v)^{p - 1}}| Kh(z,v) |^p \, e^{p \a |v|^2} dv \lesssim \int_{\R^3} |h(z, v^*)|^p\,e^{p \a |v^*|^2} dv^*.
\end{align*}
This completes the proof.
\end{proof}

For the $v$ derivative, we have the following estimates.

\begin{lemma} \label{estimate on W1pv for 1 to 2}
Let $1 \leq p < 2$ and $0 \leq \alpha < (1 - \rho)/2$, where $\rho$ is the constant in {\bf Assumption A}. Then, for $h \in W^{1, p}_\a(\O \times \R^3)$, we have
\begin{equation*}
\begin{split}
\| \nabla_v S_\O K h \|_{L^p_\a(\O \times \R^3)} \lesssim& \diam(\O)^\frac{1}{p} \| h \|_{W^{1, p}_\a(\O \times \R^3)} + \| h \|_{L^p_\a(\O \times \R^3)}\\ 
&+ \diam(\O)^{\frac{1}{p}} \| h \|_{L^p_\a(\p \O \times \R^3)}.
\end{split}
\end{equation*}
\end{lemma}

Unlike the estimate for $\nabla_x S_\O K h$, we have the term $\| h \|_{L^p_\a(\O \times \R^3)}$ in the estimate in Lemma \ref{estimate on W1pv for 1 to 2}. It appears because the function $\nabla_v k$ has a stronger singularity than $k$ so that it cannot absorb a singularity of the function $\tau(x, v)$ at $v = 0$.

\begin{proof}
By the straightforward computation, we obtain 
\begin{equation} \label{eq:SKv}
\nabla_{v}S_{\Omega}Kh = S_{\O, v} K h - (\nabla_v \nu) S_{\O, s} K h + S_\O K_v h - S_{\O,s} K \nabla_x h,
\end{equation}
where
\begin{align}
S_{\O, v} h(x, v) :=& (\nabla_{v}\tau(x, v)) e^{-\nu(v)\tau(x, v)} h(q(x,v),v), \label{Sv}\\
S_{\O, s} h(x, v) :=& \int_{0}^{\tau(x, v)} s e^{-\nu(v)s} h(x-sv,v)\,ds, \label{Ss}\\
K_v h(x, v) :=& \int_{\mathbb{R}^{3}}\nabla_{v}k(v,v^{*}) h(x,v^{*})\,dv^{*}. \label{Kv}
\end{align}

Notice that, from \cite{GuoKim}, we have 
\begin{equation} \label{tau_dv}
| \nabla_{v}\tau(x, v) | \leq \frac{| x-q(x,v) | |  n(q(x,v)) | }{|v|^{2}N_-(x,v)} = |\nabla_x \tau(x, v)| \tau(x, v).
\end{equation}
Since
\[
\tau(x, v)^p e^{-\frac{p}{2} \nu(v) \tau(x, v)} \leq \sup_{t > 0} t^p e^{-\frac{p}{2} \nu_0 t} \lesssim 1,
\]
we have 
\begin{align*}
\| S_{\O, v} K h \|_{L^p_\a(\O \times \R^3)}^p =& \int_{\Omega \times \mathbb{R}^{3}} | (\nabla_{v}\tau(x, v))e^{-\nu(v)\tau(x, v)}Kh(q(x,v),v) |^p e^{p \a |v|^2}\,dxdv\\
\lesssim& \int_{\Omega \times \mathbb{R}^{3}} |\nabla_x \tau(x, v)|^p e^{-\frac{p}{2} \nu_0\tau(x, v)}| Kh(q(x,v),v) |^p e^{p \a |v|^2}\,dxdv.
\end{align*}
We can give an estimate for the above inequality in the same way as in the proof of Lemma \ref{estimate on W1px for 1 to 2} to obtain
\begin{equation} \label{est:SKv1}
\| S_{\O, v} K h \|_{L^p_\a(\O \times \R^3)}^p \lesssim \diam(\Omega) \| h \|_{L^p_\a(\partial\Omega \times \R^3)}^p
\end{equation}
for $1 \leq p < 2$ and $0 \leq \a < (1 - \rho)/2$.

For the second and the forth term, we have
\begin{align*}
&\| S_{\O, s} K h \|_{L^1(\O \times \R^3)} \\
=&\int_{\Omega \times \mathbb{R}^{3}} \left| \int_{0}^{\tau(x, v)} s e^{-\nu(v)s} \nabla_{v}\nu(v)Kh(x-sv,v)ds \right| e^{\a |v|^2}\,dxdv\\
\lesssim& \int_{\mathbb{R}^{3}}\int_{\Omega} \left( \int_{0}^{\tau(x, v)} e^{-\frac{\nu_0}{2} s} \left| Kh(x-sv,v)\right| \,ds \right) e^{\a |v|^2} \,dxdv\\
\lesssim& \diam(\O) \| h \|_{L^1_\a(\O \times \R^3)},
\end{align*}
and
\begin{align*}
&\| S_{\O, s} K h \|_{L^p(\O \times \R^3)}^p \\
=&\int_{\Omega \times \mathbb{R}^{3}} \left| \int_{0}^{\tau(x, v)} s e^{-\nu(v)s} Kh(x-sv,v)ds \right|^p e^{p \a |v|^2}\,dxdv\\
\lesssim&\int_{\mathbb{R}^{3}}\int_{\Omega} \left( \int_{0}^{\tau(x, v)} s^{p'} e^{-\nu_0 s} \,ds \right)^{\frac{p}{p'}} \left( \int_{0}^{\tau(x, v)} e^{-\nu_0 s} \left| Kh(x-sv,v)\right|^p \,ds \right) e^{p \a |v|^2} \,dxdv
\end{align*}
for $1 < p < 2$. By Corollary \ref{cor:est_S_convex} and the proof of Lemma \ref{estimate on Lp}, we have 
\begin{align*}
\| S_{\O, s} K h \|_{L^p(\O \times \R^3)}^p \lesssim&\int_{\mathbb{R}^{3}}\int_{\Omega} \int_{0}^{\tau(x, v)} e^{-\nu_0 s} \left| Kh(x-sv,v)\right|^p e^{p \a |v|^2} \,ds \,dxdv\\
\lesssim& \diam(\O) \| h \|_{L^p_\a(\O \times \R^3)}^p.
\end{align*}
Recalling the assumption \eqref{AD'}, we have
\begin{equation} \label{est:SKv2}
\| (\nabla_v \nu) S_{\O, s} K h \|_{L^p(\O \times \R^3)} \lesssim \diam(\O)^\frac{1}{p} \| h \|_{L^p_\a(\O \times \R^3)}
\end{equation}
and
\begin{equation} \label{est:SKv4}
\| S_{\O, s} K \nabla_x h \|_{L^p(\O \times \R^3)} \lesssim \diam(\O)^\frac{1}{p} \| \nabla_x h \|_{L^p_\a(\O \times \R^3)}
\end{equation}
for $1 \leq p < 2$ and $0 \leq \a < (1 - \rho)/2$.

For the third term, by Lemma \ref{change of variable lemma}, we obtain
\begin{align*}
&\| S_\O K_v h \|_{L^1_\a(\O \times \R^3)}\\ 
=&\int_{\Omega \times \mathbb{R}^{3}} \left| \int_{0}^{\tau(x, v)}e^{-\nu(v)s} \int_{\mathbb{R}^{3}}\nabla_{v}k(v,v^{*})h(x-sv,v^{*})dv^{*}ds \right| e^{\a |v|^2}\,dxdv\\
\lesssim& \int_{\mathbb{R}^{3}} \int_{\Omega} \left| \int_{\mathbb{R}^{3}} \nabla_{v} k(u,v^{*}) h(y,v^{*})dv^{*} \right| e^{\a |v|^2} \,dydu\\
\leq& \int_{\R^3} \int_\O \left( \int_{\R^3} |\nabla_v k(u, v^*)| e^{\a |u|^2}\,du \right) |h(y, v^*)|\,dudv^*\\
\lesssim& \int_{\mathbb{R}^{3}} \int_{\Omega} |h(y,v^{*})| e^{\a |v^*|^2}\,dv^{*} \,dy
\end{align*}
and
\begin{align*}
&\| S_\O K_v h \|_{L^p_\a(\O \times \R^3)}^p\\ 
=&\int_{\Omega \times \mathbb{R}^{3}} \left| \int_{0}^{\tau(x, v)}e^{-\nu(v)s} \int_{\mathbb{R}^{3}}\nabla_{v}k(v,v^{*})h(x-sv,v^{*})dv^{*}ds \right|^p e^{p \a |v|^2}\,dxdv\\
\lesssim& \int_{\mathbb{R}^{3}} \int_{\Omega} \left| \int_{\mathbb{R}^{3}} \nabla_{v} k(u,v^{*}) h(y,v^{*}) dv^{*} \right|^p e^{p \a |u|^2} \,dydu\\
\leq& \int_{\mathbb{R}^{3}} \int_{\Omega} \left( \int_{\R^3} |\nabla_v k(u, v^*)| e^{-p' \a_1 |v^*|^2}\,dv^* \right)^{\frac{p}{p'}}\\
&\times \left( \int_{\mathbb{R}^{3}} |\nabla_{v} k(u,v^{*})| |h(y,v^{*})|^p\,e^{p \a_1 |v^*|^2} dv^{*} \right) e^{p \a |u|^2} \,dydu\\
\lesssim& \int_{\mathbb{R}^{3}} \int_{\Omega} \left( \int_{\mathbb{R}^{3}} |\nabla_{v} k(u,v^{*})| e^{p(\a - \a_1)|u|^2}\,du \right) |h(y,v^{*})|^p\,e^{p \a_1 |v^*|^2} dv^{*}\,dy\\
\lesssim& \int_{\mathbb{R}^{3}} \int_{\Omega} |h(y,v^{*})|^p\,e^{p \a |v^*|^2} dv^{*}\,dy,
\end{align*}
where $\a_1$ is the constant satisfying \eqref{alpha1_cond1} and \eqref{alpha1_cond2}. Thus, we have
\begin{equation} \label{est:SKv3}
\| S_\O K_v h \|_{L^p_\a(\O \times \R^3)} \lesssim \| h \|_{L^p_\a(\O \times \R^3)}
\end{equation}
for $1 \leq p < 2$ and $0 \leq \a < (1 - \rho)/2$.

The proof of Lemma \ref{estimate on W1pv for 1 to 2} is complete.
\end{proof}

Lemma \ref{estimate on W1p for 1 to 2} follows from Lemma \ref{estimate on Lp}, Lemma \ref{estimate on W1px for 1 to 2} and Lemma \ref{estimate on W1pv for 1 to 2}.

Now, we give a proof of Lemma \ref{BLp_est}. We start from the following estimate, which was introduced in \cite{Grisvard}: there exists a positive constant depending on $\O$ such that
\[
\| h \|_{L^p(\p \O)} \leq C (\delta^{\frac{p-1}{p}} \| \nabla_x h \|_{L^p(\O)} + \delta^{-\frac{1}{p}} \| h \|_{L^p(\O)})
\]
for all $h \in W^{1, p}(\O)$ and $0 < \delta < 1$. We replace $h$ in the above estimate by $h(\cdot, v)$ for $h \in W^{1, p}_\a(\O \times \R^3)$ to obtain
\[
\| h(\cdot, v) \|_{L^p(\p \O)} \leq C \left( \delta^{\frac{p-1}{p}} \| \nabla_x h(\cdot, v) \|_{L^p(\O)} + \delta^{-\frac{1}{p}} \| h(\cdot, v) \|_{L^p(\O)} \right)
\]
for a.e. $v \in \R^3$. Taking the $p$-th power of the above inequality, multiplying by $e^{p \a |v|^2}$ and integrating with respect to $v$, we have
\[
\| h \|_{L^p_\a(\p \O \times \R^3)}^p \leq C \left( \delta^{p-1} \| \nabla_x h \|_{L^p_\a(\O \times \R^3)}^p + \delta^{-1} \| h \|_{L^p_\a(\O \times \R^3)}^p \right),
\]
which implies the estimate in Lemma \ref{BLp_est}.

The estimate in Lemma \ref{BLp_est} does not help our argument when $p = 1$. Instead, we introduce the following estimate, which was proved in \cite{Saintier}: For any $\delta > 0$, there exists a constant $C_\delta(\O) > 0$ such that
\[
\| h \|_{L^1(\p \O)} \leq (1 + \delta) \| \nabla_x h \|_{L^1(\O)} + C_\delta(\O) \| h \|_{L^1(\O)}
\]
for all $W^{1, 1}(\O)$. In the same way as above, replacing $h$ by $h(\cdot, v)$, multiplying by $e^{\a |v|^2}$ and integrating with respect to $v$, we obtain
\[
\| h \|_{L^1_\a(\p \O \times \R^3)} \leq (1 + \delta) \| \nabla_x h \|_{L^1_\a(\O \times \R^3)} + C_\delta(\O) \| h \|_{L^1_\a(\O \times \R^3)}
\]
for all $W^{1, 1}_\a(\O \times \R^3)$. This completes the proof of Lemma \ref{BL1_est}.

As we mentioned in the introduction, we derive \eqref{SB13} from Lemma \ref{estimate on W1pv for 1 to 2} with Lemma \ref{BLp_est} and Lemma \ref{BL1_est}. Summing it from $i = 1$ to $n$, we obtain
\begin{align*}
&\frac{1}{2} \sum_{i = 0}^n \| (S_\O K)^i Jg \|_{W^{1, p}_\a(\O \times \R^3)}\\ 
\leq& \sum_{i = 1}^n \left( \| (S_\O K)^i Jg \|_{W^{1, p}_\a(\O \times \R^3)} - \frac{1}{2} \| (S_\O K)^{i - 1} Jg \|_{W^{1, p}_\a(\O \times \R^3)} \right) + \| Jg \|_{W^{1, p}_\a(\O \times \R^3)}\\
\lesssim& \| Jg \|_{W^{1, p}_\a(\O \times \R^3)} + C_2(\O) \sum_{i = 1}^n \| (S_\O K)^i Jg \|_{L^p_\a(\O \times \R^3)}.
\end{align*}
With the help of Lemma \ref{estimate on Lp}, the above estimate converges as $n \to \infty$. Thus, the series \eqref{Picard} converges in $W^{1, p}_\a(\O \times \R^3)$ for fixed $1 \leq p <2$ and $0 \leq \a < (1 - \rho)/2$, which implies the existence of the $W^{1, p}_\a$ solution assuming $Jg \in W^{1, p}_\a(\O \times \R^3)$. On the other hand, if there exists a $W^{1, p}_\a$ solution to the integral equation \eqref{integral form}, then we have $Jg = f - S_\O K f \in W^{1, p}_\a(\O \times \R^3)$. This completes the proof of the first statement in Theorem \ref{main theorem}. 

%%%%%%%%%%%%%%%%%%%%%%%%%%%%%%%%%%%%%%%%%%%%%%%%%%%%%%%%%%%%%%%%%%%%%%%%%%%%%%%%%%%%%%%%%%%%%%%%%%%%%%%
\section{Regularity on small convex domains of positive Gaussian curvature} \label{sec:bounded_domain_B}
%%%%%%%%%%%%%%%%%%%%%%%%%%%%%%%%%%%%%%%%%%%%%%%%%%%%%%%%%%%%%%%%%%%%%%%%%%%%%%%%%%%%%%%%%%%%%%%%%%%%%%%

In this section, we deal with the second statement of Theorem \ref{main theorem}. It suffices to prove the estimate \eqref{SB13} for the case $2 \leq p < 3$. We cannot treat this case in the same way as Lemma \ref{estimate on W1px for 1 to 2} due to the singularity of $N^{-1}(z, v)$. In order to control it, we make use of Lemma \ref{circle lemma}. Eventually, we prove the following estimate.

\begin{lemma} \label{estimate on W1px for 2 to 3}
Let $2 \leq p < 3$ and $0 \leq \alpha < (1 - \rho)/2$, where $\rho$ is the constant in {\bf Assumption A}. Also, let $\O$ be a $C^2$ bounded convex domain of positive Gaussian curvature. Then, for $h \in W^{1, p}_\a(\O \times \R^3)$, we have
\[
\| \nabla_x S_\O K h \|_{L^p_\a(\O \times \R^3)} \lesssim \diam(\O)^\frac{1}{p} \| h \|_{W^{1, p}_\a(\O \times \R^3)} + C_3(\O)^{\frac{1}{p}} \| h \|_{L^p_\a(\partial \O \times \R^3)},
\]
where $C_3(\O)$ is the constant in Lemma \ref{circle lemma}.
\end{lemma}

\begin{proof}
As in the case of Lemma \ref{estimate on W1px for 1 to 2}, we already have 
\[
\| S_\O K \nabla_x h \|_{L^p_\a(\O \times \R^3)} \lesssim \diam(\O)^{\frac{1}{p}} \| \nabla_x h \|_{L^p_\a(\O \times \R^3)}
\] 
for $2 \leq p < 3$ and $0 \leq \a < (1 - \rho)/2$. Thus, in what follows, we give an estimate for $S_{\O, x} K h$. In particular, we modify the estimate \eqref{est:Lp_domain_to_boundary} for $2 \leq p < 3$. Using the estimate
\[
\int_0^{\tau(z, -v)} e^{-p \nu_0 s}\,ds \leq \tau(z, -v) = \frac{|z - q(z, -v)|}{|v|}
\]
and Lemma \ref{circle lemma}, we get
\begin{align*}
&\int_{\mathbb{R}^{3}}\int_{\Omega} | (\nabla_x \tau(x, v)) e^{-\nu(v)\tau(x, v)}Kh(q(x,v),v) |^p e^{p \a |v|^2} \,dxdv\\
\lesssim&\int_{\mathbb{R}^{3}}\frac{1}{|v|^p}\int_{\Gamma^-_v} \frac{1}{N(z,v)^{p - 1}}| Kh(z,v) |^p | z-q(z,-v) | e^{p \a |v|^2}\,d\Sigma(z)dv\\
\lesssim&C_3(\O) \int_{\mathbb{R}^{3}}\frac{1}{|v|^p} \int_{\Gamma^-_v} \frac{1}{N(z,v)^{p - 2}}| Kh(z,v) |^p e^{p \a |v|^2}\,d\Sigma(z)dv.
\end{align*}

We first give an estimate for $2 < p < 3$. In this case, since $3/2 < p' < 2$, we can apply Corollary \ref{cor:est_K_alpha} to obtain
\begin{align*}
| Kh(z,v) |^p \leq& \left( \int_{\R^3} |k(v, v^*)|^{p'} e^{-p' \a |v^*|}\,dv^* \right)^{\frac{p}{p'}} \left( \int_{\R^3} |h(z, v^*)|^p e^{p \a |v^*|^2}\,dv^* \right)\\
\lesssim& (1 + |v|)^{-(p - 1)} e^{- p \a |v|^2} \int_{\R^3} |h(z, v^*)|^p e^{p \a |v^*|^2}\,dv^*.
\end{align*}
Hence, we have
\begin{align*}
&\int_{\mathbb{R}^{3}} \frac{1}{|v|^p} \int_{\Gamma^-_v} \frac{1}{N(z,v)^{p - 2}}| Kh(z,v) |^p e^{p \a |v|^2}\,d\Sigma(z)dv\\
\lesssim& \int_{\partial \O} \left( \int_{\Gamma^-_z} \frac{1}{|v|^p} \frac{1}{(1 + |v|)^{p - 1}} \frac{1}{N(z,v)^{p - 2}}\,dv \right) \left( \int_{\R^3} |h(z, v^*)|^p e^{p \a |v^*|^2}\,dv^* \right)\,d\Sigma(z).
\end{align*}

In what follows, we show that the inner integral with respect to $v$ is uniformly bounded with respect to $z$. For fixed $z \in \partial\Omega$, we introduce the spherical coordinates $v= (r \sin\theta \cos\phi, r \sin\theta \cos\phi, r \cos\theta)$ so that $\theta=0$ corresponds to the $-n(z)$ direction. Then, we have
\begin{align*}
&\int_{\mathbb{R}^{3}} \frac{1}{|v|^{p}} \frac{1}{N(z,v)^{p - 2}} \frac{1}{(1 + |v|)^{p - 1}}\,dv\\ 
=& 2 \pi \int_0^\infty \left( \int_0^{\pi/2} \frac{1}{r^p} \frac{1}{\cos^{p-2}\theta} \frac{1}{(1 + r)^{p - 1}}r^2 \sin\theta\,d\theta \right)\,dr\\
=& 2\pi \left( \int_0^\infty \frac{1}{r^{p - 2}} \frac{1}{(1 + r)^{p - 1}} \,dr \right) \left( \int_0^1 \frac{1}{t^{p-2}}\,dt \right).
\end{align*}
Here, we change a variable of integration $t = \cos \theta$. Since $2 < p < 3$, we have $0 < p - 2 < 1$, $p - 2 + p - 1 = 2 p - 3 > 1$, and $0 \leq p - 2 < 1$. Thus, we have
\[
\left( \int_0^\infty \frac{1}{r^{p - 2}} \frac{1}{(1 + r)^{p - 1}} \,dr \right) \left( \int_0^1 \frac{1}{t^{p-2}}\,dt \right) \lesssim 1.
\]
We remark that the above estimate is independent of the choice of the point $z \in \p \O$. Therefore, we have
\begin{align*}
&\int_{\mathbb{R}^{3}}\int_{\Omega} | (\nabla_x \tau(x, v)) e^{-\nu(v)\tau(x, v)}Kh(q(x,v),v) |^p e^{p \a |v|^2}\,dxdv\\ 
\lesssim& C_3(\O) \int_{\mathbb{R}^{3}}\int_{\partial\Omega}|h(z,v^{*})|^p e^{p \a |v^*|^2}\,d\Sigma(z)dv^{*}.
\end{align*}

The above argument does not hold for $p = 2$ since $2 p - 3 = 1$, which causes the divergence of the $r$ integral. In order to treat this case, we consider a different estimate. We note that
\begin{align*}
|Kh(z, v)|^2 \leq& \left( \int_{\R^3} |k(v, v^*)|^{\frac{3}{2}} e^{-2 \a_2 |v^*|^2}\,dv^* \right)\\
&\times \left( \int_{\R^3} |k(v, v^*)|^{\frac{1}{2}} |h(z, v^*)|^2 e^{2 \a_2 |v^*|^2}\,dv^* \right)\\ 
\lesssim& \left( \int_{\R^3} |k(v, v^*)|^{\frac{1}{2}} |h(z, v^*)|^2 e^{2 \a_2 |v^*|^2}\,dv^* \right) e^{-2 \a_2 |v|^2},
\end{align*}
where $\a_2$ is the constant satisfying
\begin{equation} \label{alpha2_cond1}
-\frac{3(1 - \rho)}{4} < - 2 \a_2 < \frac{3(1 - \rho)}{4}
\end{equation}
and
\begin{equation} \label{alpha2_cond2}
-\frac{1 - \rho}{4} < 2 (\a - \a_2) < \frac{1 - \rho}{4}.
\end{equation}
We can show the existence of such a constant $\a_2$ in the same way as for $\a_1$ satisfying \eqref{alpha1_cond1} and \eqref{alpha1_cond2}. Thus, we have
\begin{align*}
&\int_{\mathbb{R}^{3}} \frac{1}{|v|^2} \int_{\Gamma^-_v} | Kh(z,v) |^2 e^{2 \a |v|^2}\,d\Sigma(z)dv\\
\lesssim& \int_{\mathbb{R}^{3}} \frac{1}{|v|^2} \int_{\Gamma^-_v} \left( \int_{\R^3} |k(v, v^*)|^{\frac{1}{2}} |h(z, v^*)|^2 e^{2 \a_2 |v^*|^2}\,dv^* \right) e^{2 (\a - \a_2) |v|^2}\,d\Sigma(z)dv\\
\leq& \int_{\R^3} \int_{\partial \O} \left( \int_{\R^3} \frac{1}{|v|^2} |k(v, v^*)|^{\frac{1}{2}} e^{2 (\a - \a_2) |v|^2}\,dv \right) |h(z, v^*)|^2 e^{2 \a_2 |v^*|^2}\,d\Sigma(z)dv^*\\
\lesssim& \int_{\R^3} \int_{\partial \O} |h(z, v^*)|^2 e^{2 \a |v^*|^2}\,d\Sigma(z)dv^*.
\end{align*}
This completes the proof.
\end{proof}

For the $v$ derivative, we have the following estimate.

\begin{lemma} \label{estimate on W1pv for 2 to 3}
Let $2 \leq p < 3$ and $0 \leq \alpha < (1 - \rho)/2$, where $\rho$ is the constant in {\bf Assumption A}. Also, let $\O$ be a $C^2$ bounded convex domain of positive Gaussian curvature. Then, for $h \in W^{1, p}_\a(\O \times \R^3)$, we have

\begin{align*}
\| \nabla_v S_\O K h \|_{L^p_\a(\O \times \R^3)} \lesssim& \diam(\O)^\frac{1}{p} \| h \|_{W^{1, p}_\a(\O \times \R^3)} + \| h \|_{L^p_\a(\O \times \R^3)}\\
&+ C_3(\O)^{\frac{1}{p}} \| h \|_{L^p_\a(\partial \O \times \R^3)},
\end{align*}
where $C_3(\O)$ is the constant in Lemma \ref{circle lemma}.
\end{lemma}

\begin{proof}
Recall the formula \eqref{eq:SKv}. We can see that the estimates \eqref{est:SKv2}, \eqref{est:SKv4} and \eqref{est:SKv3} hold for $2 \leq p < 3$ and $0 \leq \a < (1 - \rho)/2$. For the first term, by Lemma \ref{estimate on W1px for 2 to 3}, we have
\[
\| S_{\O, v} K h \|_{L^p_\a(\O \times \R^3)} \lesssim C_3(\O) \| h \|_{L^p_\a(\p \O \times R^3)}.
\]
Therefore, Lemma \ref{estimate on W1pv for 2 to 3} is proved.
\end{proof}

Lemma \ref{estimate on W1p for 2 to 3} follows from Lemma \ref{estimate on Lp}, Lemma \ref{estimate on W1px for 2 to 3} and Lemma \ref{estimate on W1pv for 2 to 3}.

Combining Lemma \ref{estimate on W1p for 2 to 3} with Lemma \ref{BLp_est}, we derive the estimate \eqref{SB13} for fixed $2 \leq p <3$ and $0 \leq \a < (1 - \rho)/2$, and this proves the convergence of the series \eqref{Picard} in $W^{1, p}_\a(\O \times \R^3)$ and the existence of the solution for the case where $Jg \in W^{1, p}_\a(\O \times \R^3)$. On the other hand, if the solution $f$ belongs to $W^{1, p}_\a (\Omega\times\mathbb{R}^{3})$, then $S_{\Omega}Kf$ also belongs to $W^{1, p}_\a (\Omega\times\mathbb{R}^{3})$, which with $f=Jg+S_{\Omega}Kf$ will make $Jg$ belong to $W^{1, p}_\a(\Omega\times\mathbb{R}^{3})$. This completes the proof of the second statement in Theorem \ref{main theorem}.

%%%%%%%%%%%%%%%%%%%%%%%%%%%%%%%%%%%%%%%%%%%%%%%%%%%%%%%%%%%%%%%%%%%%%%%%%%%%%%%%%%%%%%%%%%%%%%%%%%%%%%%
\section{Construction of counterexamples} \label{sec:counterexample}
%%%%%%%%%%%%%%%%%%%%%%%%%%%%%%%%%%%%%%%%%%%%%%%%%%%%%%%%%%%%%%%%%%%%%%%%%%%%%%%%%%%%%%%%%%%%%%%%%%%%%%%
\subsection{A counterexample for $p=2$}

In this subsection, we give a proof of Lemma \ref{lem:optimality2}. In other words, we shall show that, for fixed $1 \leq p < 2$ and $0 \leq \a < (1 - \rho)/2$, the solution to the boundary value problem \eqref{SLBE}-\eqref{inbdry} with the boundary data \eqref{inbdry_optimal_convex_2} belongs to $L^2_\a(\O \times \R^3) \cap W^{1, p}_\a(\O \times \R^3)$ for sufficiently small $\diam(\O)$ but this solution does not belong to $W^{1, 2}_\a(\O \times \R^3)$.

Let $\O$ be a bounded convex domain with partially flat boundary as described in Section \ref{sec:intro} and let the boundary data $g$ be given by \eqref{inbdry_optimal_convex_2}. We see that $g \in L^p_\alpha(\Gamma^-) \cap C_\alpha(\Gamma^-)$ for all $1 \leq p < \infty$ and $0 \leq \alpha < 1/2$. Thus, by Proposition \ref{prop:bound_J_alpha} and Proposition \ref{prop:bound_J_Calpha}, we have $Jg \in L^p_\alpha(\O \times \R^3) \cap C_\alpha((\O \times \R^3) \cup \Gamma^\pm)$ for all $1 \leq p < \infty$ and $0 \leq \alpha < (1 - \rho)/2$. Therefore, by Lemma \ref{estimate on Lp} and Lemma \ref{lem:small_second_convex}, the integral equation \eqref{integral form} has the unique solution $f$ in $L^p_\alpha(\O \times \R^3) \cap C_\alpha((\O \times \R^3) \cup \Gamma^\pm)$ if $\diam(\O)$ is sufficiently small. In particular, we see that $f \in L^2_\a(\O \times \R^3)$.

We show that $Jg \in W^{1, p}_\a(\O \times \R^3)$ for all $0 \leq \a < (1- \rho)/2$ and $1 \leq p < 2$. By differentiation, we have
\[
\nabla_x Jg = -\nu(v) (\nabla_x \tau(x, v)) Jg + (\nabla_x q) J(\nabla_X g),
\]
where $\nabla_X$ is the covariant derivative on $\p \O$. We note that $\nabla_x q(x, v)$ is a matrix of the form
\[
\nabla_x q(x, v) = I_3 - (\nabla_x \tau(x, v)) \otimes v.
\]
Here, $I_3$ is the identity matrix of order $3$. In what follows, we investigate singularity of the gradient $\nabla_x \tau(x, v)$.

\begin{lemma} \label{lem:tau_dx_flat}
Let $\Omega$ be a bounded convex domain satisfying \eqref{flat_boundary} and \eqref{halfball_contained}. Also, let $p \geq 1$. Then, the integral
\[
\int_{\O \times \R^3} |\nabla_x \tau(x, v)|^p |v|^b e^{-a |v|^2}\,dxdv
\]
converges for all $a > 0$ and $b \geq 0$ if and only if $p < 2$.
\end{lemma}

\begin{proof}
From \eqref{domain_to_boundary}, we have
\begin{align*}
&\int_{\O \times \R^3} |\nabla_x \tau(x, v)|^p |v|^b e^{-a p |v|^2}\,dxdv\\ 
=& \int_{\partial \O} \int_{\Gamma^-_z} \int_0^{\tau(z, -v)} |\nabla_x \tau(z + tv, v)|^p |v|^b e^{-a p |v|^2}\,dt N(z, v) |v|\,dvd\Sigma(z).
\end{align*}
From \eqref{tau_dx}, we see that 
\[
|\nabla_x \tau(z + tv, v)| = |\nabla_x \tau(z, v)| = \frac{1}{N(z, v)|v|}.
\] 
Thus, we have
\begin{align*}
&\int_{\partial \O} \int_{\Gamma^-_z} \int_0^{\tau(z, -v)} |\nabla_x \tau(z + tv, v)|^p |v|^b e^{-a p |v|^2}\,dt N(z, v) |v|\,dv d\Sigma(z)\\ 
=& \int_{\partial \O} \int_{\Gamma^-_z} \frac{\tau(z, -v)}{N(z, v)^{p-1} |v|^{p-b-1}} e^{-a p |v|^2}\,dv d\Sigma(z).
\end{align*}

Thanks to the boundedness of the domain, we have
\[
\tau(z, -v) \leq \frac{\diam(\O)}{|v|}, 
\]
and hence we have
\begin{align*}
&\int_{\partial \O} \int_{\Gamma^-_z} \int_0^{\tau(z, -v)} |\nabla_x \tau(z + tv, v)|^p |v|^b e^{-a p |v|^2}\,dt N(z, v) |v|\,dv d\Sigma(z)\\ 
\leq& \diam(\O) \int_{\partial \O} \int_{\Gamma^-_z} \frac{1}{N(z, v)^{p-1} |v|^{p-b}} e^{-a p |v|^2}\,dv d\Sigma(z).
\end{align*}

We fix the point $z \in \partial \O$ and compute the inner integral with respect to $v$. We introduce the spherical coordinates for $v$ so that $\theta = 0$ corresponding to the $-n(z)$ direction. Then, we have
\begin{align*}
\int_{\Gamma^-_z} \frac{1}{N(z, v)^{p-1} |v|^{p-b}} e^{-a p |v|^2}\,dv =& 2\pi \int_0^{\frac{\pi}{2}} \frac{\sin \theta}{\cos^{p-1} \theta}\,d\theta \int_0^\infty \rho^{2+b-p} e^{-a p \rho^2}\,d\rho\\
=& 2\pi \int_0^1 t^{1-p}\,dt \int_0^\infty \rho^{2+b-p} e^{-a p \rho^2}\,d\rho.
\end{align*}
The above integral converges for all $a > 0$ and $b \geq 0$ if $p < 2$.

On the other hand, under the assumptions \eqref{flat_boundary} and \eqref{halfball_contained}, we have
\begin{equation} \label{tau_flat}
\tau(z, -v) \geq \frac{r_1}{2|v|} 
\end{equation}
for $(z, v) \in \Gamma^-$ with $z \in D_{r_1/2}$. Hence, we have
\begin{align*}
&\int_{\O \times \R^3} |\nabla_x \tau(x, v)|^p |v|^b e^{-a p |v|^2}\,dxdv\\ 
\geq& \frac{r_1}{2} \int_{D_{r_1/2}} \int_{\Gamma^-_z} \frac{1}{N(z, v)^{p-1} |v|^{p-b}} e^{-a p |v|^2}\,dv d\Sigma(z). 
\end{align*}
In the same way as above, we can show that the right hand side diverges for all $a > 0$ and $b \geq 0$ if $p \geq 2$. 

Therefore, the integral we concern converges for all $a > 0$ and $b \geq 0$ if and only if $p < 2$. This completes the proof. 
\end{proof}

Let $g$ be the function on $\Gamma^-$ defined by \eqref{inbdry_optimal_convex_3}. Then, we have
\begin{equation} \label{Jg_dx_1}
|-\nu(v) (\nabla_x \tau(x, v)) Jg(x, v) e^{\a |v|^2}| \leq \nu_1 (1 + |v|) |\nabla_x \tau(x, v)| e^{-\left( \frac{1}{2} - \a \right) |v|^2} 
\end{equation}
and
\begin{equation} \label{Jg_dx_2}
|(\nabla_x q) J(\nabla_X g)(x, v) e^{\a |v|^2}| \leq (1 + |\nabla_x \tau(x, v)| |v|) \| \nabla_X \varphi_1 \|_{L^\infty(\p \O)} e^{- \left(\frac{1}{2} - \a \right)|v|^2}.
\end{equation}
Thus, by Lemma \ref{lem:tau_dx_flat}, we see that $\nabla_x Jg \in L^p_\a(\O \times \R^3)$ for all $0 \leq \a < (1 - \rho)/2$ and $1 \leq p < 2$.

For the $v$ derivative, we have
\[
\nabla_v Jg = - \tau(x, v) (\nabla_v \nu(v)) Jg - \nu(v) (\nabla_v \tau(x, v)) Jg + (\nabla_v q) J (\nabla_X g) + J(\nabla_v g). 
\]
The matrix $\nabla_v q$ reads
\[
\nabla_v q(x, v) = - \tau(x, v) I_3 - (\nabla_v \tau(x, v)) \otimes v.
\]

For the function $g$ defined by \eqref{inbdry_optimal_convex_2}, recalling \eqref{tau_dv}, we have
\begin{align}
&|- \tau(x, v) (\nabla_v \nu(v)) Jg(x, v) e^{\a |v|^2}| \lesssim e^{-\left( \frac{1}{2} - \a \right) |v|^2}, \label{Jg_dv_1}\\
&|- \nu(v) (\nabla_v \tau(x, v)) Jg(x, v) e^{\a |v|^2}| \lesssim (1 + |v|) |\nabla_x \tau(x, v)| e^{-\left( \frac{1}{2} - \a \right) |v|^2}, \label{Jg_dv_2}\\
&|(\nabla_v q) J (\nabla_X g)(x, v) e^{\a |v|^2}| \lesssim (1 + |\nabla_x \tau(x, v)| |v|) \| \nabla_X \varphi_1 \|_{L^\infty(\p \O)} e^{-\left( \frac{1}{2} - \a \right) |v|^2}, \label{Jg_dv_3}\\
&|J(\nabla_v g) e^{\a |v|^2}| \lesssim |v| e^{-\left( \frac{1}{2} - \a \right) |v|^2}. \label{Jg_dv_4}
\end{align}
Thus, we apply Lemma \ref{lem:tau_dx_flat} again to obtain that $\nabla_v Jg$ belongs to $L^p_\a(\O \times \R^3)$ for all $0 \leq \a < (1 - \rho)/2$ and $1 \leq p < 2$. Therefore, we have checked that the function $Jg$ with the boundary data $g$ defined by \eqref{inbdry_optimal_convex_2} belongs to $W^{1, p}_\a(\O \times \R^3)$ for all $0 \leq \a < (1 - \rho)/2$ and $1 \leq p < 2$ and, if $\diam(\O)$ is small enough, the corresponding solution $f$ to the boundary value problem \eqref{SLBE}-\eqref{inbdry} also belongs to $W^{1, p}_\a(\O \times \R^3)$.

In what follows, we show that the solution $f$ does not belong to $W^{1, 2}_\a(\O \times \R^3)$ for all $0 \leq \a < (1 - \rho)/2$. To this aim, we assume that the solution belongs to $W^{1, 2}_\a(\O \times \R^3)$, and derive a contradiction. We formally differentiate the equation \eqref{integral form} with respect to $x$ to obtain
\begin{equation*}
\begin{split}
\nabla_x f(x, v) =& -\nu(v) (\nabla_x \tau(x, v)) Jg(x, v) + (\nabla_x q(x, v)) J (\nabla_X g)(x, v)\\
&+ S_{\O, x} K f(x, v) + S_\O K (\nabla_x f) (x, v).
\end{split}
\end{equation*}
By assumption, we see that $S_\O K (\nabla_x f) \in L^2_\a(\O \times \R^3)$, and therefore the integral
\begin{equation} \label{eq:L2_part_of_dx}
\int_{\O \times \R^3} \left| \nabla_xf-S_\O K \nabla_x f \right|^2 e^{2\a |v|^2}\,dxdv
\end{equation}
is bounded. In what follows, we shall show that the integral \eqref{eq:L2_part_of_dx} is in fact unbounded, which causes the contradiction.

Let $r_2 > 0$ and 
\begin{equation} \label{def:Dr_flat}
D_{r_1, r_2} := \{ (x, v) \in \O \times \R^3 \mid q(x, v) \in D_{r_1/4}, \tau(x, v) \leq 1, |v| < r_2 \}.
\end{equation}
In this region, we have $J (\nabla_X g) = 0$ and
\begin{align*}
S_{\O, x}Kf (x, v) =& (\nabla_x \tau(x, v)) e^{-\nu(v)\tau(x, v)} \int_{\Gamma^+_{q(x, v)}} k(v, v^*) f(q(x, v), v^*)\,dv^*\\ 
&+ (\nabla_x \tau(x, v)) e^{-\nu(v)\tau(x, v)} \int_{\Gamma^-_{q(x, v)}} k(v, v^*) e^{-\frac{1}{2} |v^*|^2}\,dv^*.
\end{align*}
We substitute the function $f$ in the first term of the right hand side by the integral equation \eqref{integral form} again to obtain
\begin{align*}
\int_{\Gamma^+_{q(x, v)}} k(v, v^*) f(q(x, v), v^*)\,dv^*
=& \int_{\Gamma^+_{q(x, v)}} k(v, v^*) Jg(q(q(x, v), v^*), v^*)\,dv^*\\
&+ \int_{\Gamma^+_{q(x, v)}} k(v, v^*) S_\O K f(q(q(x, v), v^*), v^*)\,dv^*.
\end{align*}
Here, since $q(q(x, v), v^*) \notin D_{r_1}$ for $(x, v) \in D_{r_1, r_2}$ and $v^* \in \Gamma^+_{q(x, v)}$, the first term in the right hand side is zero. On the other hand, by Lemma \ref{lem:small_second_convex} and \eqref{eq:uniform_est}, we have
\begin{align*}
\left| \int_{\Gamma^+_{q(x, v)}} k(v, v^*) S_\O K f(q(q(x, v), v^*), v^*)\,dv^* \right| \lesssim& \| f \|_{C_\alpha((\O \times \R^3) \cup \Gamma^\pm)} \diam(\O)\\
\lesssim& \diam(\O).
\end{align*}
Therefore, we can make contribution from the integral
\[
\int_{\Gamma^+_{q(x, v)}} k(v, v^*) f(q(x, v), v^*)\,dv^*
\]
arbitrary small by taking $\diam(\O)$ sufficiently small.

In what follows, we consider the case $\gamma = 1$, which corresponds to the hard-sphere model. In this case, the coefficient $\nu(v)$ and the integral kernel $k$ have the following explicit formulae:
\[
\nu(v) = 2^{-\frac{3}{2}} \left[ e^{-|v|^2} + \left( 2|v| + \frac{1}{|v|} \right) \int_0^{|v|} e^{- \eta^2}\,d\eta \right]
\]
and
\begin{align*}
k(v, v^*) =& 2^{-\frac{3}{2}} \pi^{-1} \left\{ 2|v^* - v|^{-1} \exp \left( -\frac{1}{4} \frac{(|v^*|^2 - |v|^2)^2}{|v^* - v|^2} -\frac{1}{4} |v^* - v|^2 \right) \right.\\ 
&\left. - |v^* - v| \exp \left( -\frac{1}{2}(|v^*|^2 + |v|^2) \right) \right\}.
\end{align*}

\begin{lemma} \label{lem:neg_int_flat}
There exist $\eta_0 > 0$ and $r_2 > 0$ such that 
\begin{equation} \label{ineq:nonvanish_flat}
\nu(v) e^{-\frac{1}{2}|v|^2} - \int_{\Gamma^-_{q(x, v)}} k(v, v^*) e^{-\frac{1}{2} |v^*|^2}\,dv^* > \eta_0
\end{equation}
for all $(x, v) \in D_{r_1, r_2}$.
\end{lemma}

\begin{proof}
Notice that 
\[
\Gamma^-_{q(x, v)} = \{ v^* = (v_1, v_2, v_3) \in \R^3 \mid v^*_1 < 0 \}
\]
for $(x, v) \in D_{r_1, r_2}$.

By the continuity of the left hand side of \eqref{ineq:nonvanish_flat}, it suffices to show positivity at the limit $|v| \to 0$. In this limit, we have
\[
\lim_{|v| \to 0} \nu(v) e^{-\frac{1}{2} |v|^2} = \nu(0) = 2^{-\frac{3}{2}} (1 + 1) = 2^{-\frac{1}{2}}
\]
and
\begin{align*}
& \lim_{|v| \to 0} \int_{\Gamma^-_{q(x, v)}} k(v, v^*) e^{-\frac{1}{2}|v^*|^2}\,dv^*\\
=& \int_{\{ v^*_1 < 0 \}} k(0, v^*) e^{-\frac{1}{2}|v^*|^2}\,dv^*\\
=& \int_{\{ v^*_1 < 0 \}} 2^{-\frac{3}{2}} \pi^{-1} \left\{ 2 |v^*|^{-1} e^{-\frac{1}{2}|v^*|^2} - |v^*| e^{-\frac{1}{2} |v^*|^2 } \right\} e^{-\frac{1}{2}|v^*|^2}\,dv^*.
\end{align*}

We introduce the spherical coordinates; $v^* = (-\rho \cos \theta, \rho \sin \theta \cos \phi, \rho \sin \theta \sin \phi )$ for $0 < \theta < \pi/2$ and $-\pi < \phi < \pi$. Then, we have
\begin{align*}
&\int_{\Gamma^-_{q(x, \hat{v})}} 2^{-\frac{3}{2}} \pi^{-1} \left\{ 2 |v^*|^{-1} e^{-\frac{1}{2}|v^*|^2} - |v^*| e^{-\frac{1}{2} |v^*|^2 } \right\} e^{-\frac{1}{2}|v^*|^2}\,dv^*\\
=& 2^{-\frac{3}{2}} \pi^{-1} \int_0^\infty \left( \int_0^{\pi/2} \left( \int_0^{2\pi} (2 \rho^{-1} - \rho) e^{-\rho^2} \rho^2 \sin \theta\,d\varphi \right)\,d\theta \right)\,d\rho\\
=& 2^{-\frac{1}{2}} \int_0^\infty (2 \rho - \rho^3) e^{-\rho^2} \,d\rho.
\end{align*}
Since
\[
\int_0^\infty \rho^3 e^{-\rho^2} \,d\rho = - \frac{1}{2} \int_0^\infty \rho^2 \frac{d}{d\rho} e^{-\rho^2}\,d\rho = \int_0^\infty \rho e^{-\rho^2}\,d\rho,
\]
we have
\[
\int_0^\infty (2 \rho - \rho^3) e^{-\rho^2} \,d\rho = \int_0^\infty \rho e^{-\rho^2}\,d\rho = 2^{-1}.
\]
Therefore, we have
\[
\int_{\{ v^*_1 < 0 \}} k(0, v^*) e^{-\frac{1}{2}|v^*|^2}\,dv^* = 2^{-\frac{3}{2}}
\]
and
\[
\nu(0) - \int_{\{ v^*_1 < 0 \}} k(0, v^*) e^{-\frac{1}{2}|v^*|^2}\,dv^* = 2^{-\frac{1}{2}} - 2^{-\frac{3}{2}} = 2^{-\frac{3}{2}}.
\]
This completes the proof.
\end{proof}

We are ready to prove that the derivative $\nabla_x f$ does not belong to $L^2_\a(\O \times \R^3)$ if $r_1$, $r_2$ and $\diam(\O)$ are sufficiently small. Thanks to Lemma \ref{lem:neg_int_flat}, we have
\begin{align*}
\int_{\O \times \R^3} \left| \nabla_xf-S_\O K \nabla_x f \right|^2 e^{2\a |v|^2}\,dxdv \gtrsim \int_{D_{r_1, r_2}} |\nabla_x \tau(x, v)|^2\,dxdv.
\end{align*}
Here, we perform the same change of variable as \eqref{domain_to_boundary}. We notice that, because of the restriction $\tau(x, v) < 1$, the integral has a different form:
\begin{align*}
&\int_{D_{r_1, r_2}} |\nabla_x \tau(x, v)|^2\,dxdv\\ 
=& \int_{D_{r_1/4}} \int_{\{ v_1 < 0 \} \cap \{ |v| < r_2 \}} \int_0^{\min\{\tau(z, -v), 1\}} |\nabla_x \tau(z + tv, v)|^2\,dt N(z, v) |v|dv\,d\Sigma(z).
\end{align*}
From \eqref{tau_dx}, we have
\begin{align*}
\int_0^{\min\{\tau(z, -v), 1\}} |\nabla_x \tau(z + tv, v)|^2\,dt N(z, v) |v| = \frac{\min\{\tau(z, -v), 1\}}{N(z, v)|v|}
\end{align*}
We restrict ourselves to the case $|v| < r_1/2$. In this case, we have $\tau(z, -v) > 1$. Let $r_3 := \min \{r_1/2, r_2\}$. Then, we have
\begin{align*}
&\int_{D_{r_1/4}} \int_{\{ v_1 < 0 \} \cap \{ |v| < r_2 \}} \int_0^{\min\{\tau(z, -v), 1\}} |\nabla_x \tau(z + tv, v)|^2\,dt N(z, v) |v|dv\,d\Sigma(z)\\ 
\geq& \int_{D_{r_1/4}} \int_{\{ v_1 < 0 \} \cap \{ |v| < r_3 \}} \frac{1}{N(z, v)|v|}\,dv\,d\Sigma(z).
\end{align*}
Introducing the spherical coordinates to
 $v$ so that $\theta = 0$ corresponds to $(-1, 0, 0)$, we have
\begin{align*}
\int_{\{ v_1 < 0 \} \cap \{ |v| < r_3 \}} \frac{1}{N(z, v)|v|}\,dv =& \pi r_3^2 \int_0^{\pi/2} \frac{\sin \theta}{\cos \theta}\,d\theta, 
\end{align*}
which is divergent for all $z \in D_{r_1/4}$. Therefore the integral \eqref{eq:L2_part_of_dx} is not bounded, which leads to the contradiction. 

\subsection{A counterexample for $p=3$}

In this subsection, we give a proof of Lemma \ref{lem:optimality3}.

%Next, we introduce a boundary value problem \eqref{SLBE}-\eqref{inbdry} whose solution does not belongs to $W^{1, 3}_\a(\O \times \R^3)$ but belongs to $L^3_\a(\O \times \R^3) \cap W^{1, p}_\a(\O \times \R^3)$ for all $1 \leq p < 3$ and $0 \leq \a < (1 - \rho)/2$. 

Let $\O$ be a ball centered at the origin with radius $r$. We pose the boundary data $g$ as \eqref{inbdry_optimal_convex_3}. Notice that $g \in L^p_\alpha(\Gamma^-) \cap C_\alpha(\Gamma^-)$ for all $1 \leq p < \infty$ and $0 \leq \alpha < 1/2$. Thus, by Proposition \ref{prop:bound_J_alpha} and Proposition \ref{prop:bound_J_Calpha}, we have $Jg \in L^p_\alpha(\O \times \R^3) \cap C_\alpha((\O \times \R^3) \cup \Gamma^\pm)$ for all $1 \leq p < \infty$ and $0 \leq \alpha < (1 - \rho)/2$. Therefore, by Lemma \ref{estimate on Lp} and Lemma \ref{lem:small_second_convex}, the integral equation \eqref{integral form} has the unique solution $f$ in $L^p_\alpha(\O \times \R^3) \cap C_\alpha((\O \times \R^3) \cup \Gamma^\pm)$. Especially, we have $f \in L^3_\a(\O \times \R^3)$ for all $0 \leq \a < (1 - \rho)/2$.

Corresponding to Lemma \ref{lem:tau_dx_flat}, we have the following estimate.

\begin{lemma} \label{lem:tau_dx_ball}
Let $\Omega$ be a ball with its center the origin and its radius $r$. Also, let $p \geq 1$. Then, the integral
\[
\int_{\O \times \R^3} |\nabla_x \tau(x, v)|^p |v|^b e^{-a |v|^2}\,dxdv
\]
converges for all $a > 0$ and $b \geq 0$ if and only if $p < 3$.
\end{lemma}

\begin{proof}
As we computed in the proof of Lemma \ref{lem:tau_dx_flat}, we have
\begin{align*}
&\int_{\O \times \R^3} |\nabla_x \tau(x, v)|^p |v|^b e^{-a p |v|^2}\,dxdv\\ 
=&\int_{\partial \O} \int_{\Gamma^-_z} \int_0^{\tau(z, -v)} |\nabla_x \tau(z + tv, v)|^p |v|^b e^{-a p |v|^2}\,dt N(z, v) |v|\,dv d\Sigma(z)\\ 
=& \int_{\partial \O} \int_{\Gamma^-_z} \frac{\tau(z, -v)}{N(z, v)^{p-1} |v|^{p-b-1}} e^{-a p |v|^2}\,dv d\Sigma(z).
\end{align*}

If the domain $\O$ is the ball with its center the origin and its radius $r$, the function $\tau(x, v)$ is explicitly described as 
\begin{equation} \label{tau_ball}
\tau(x, v) = \frac{1}{|v|} \left( x \cdot \hat{v} + \sqrt{r^2 - |x|^2 + (x \cdot \hat{v})^2} \right).
\end{equation}
From \eqref{tau_ball}, we also see that
\begin{equation} \label{eq:tau_ball_ex}
\tau(z, -v) = \frac{2 |\hat{v} \cdot z|}{|v|} = \frac{2r}{|v|} N(z, v), \quad (z, v) \in \Gamma^-.
\end{equation}
Therefore, we have
\[
\int_{\O \times \R^3} |\nabla_x \tau(x, v)|^p |v|^b e^{-a p |v|^2}\,dxdv = 2r \int_{\p \O} \int_{\Gamma^-_z} \frac{1}{N(z, v)^{p-2} |v|^{p-b}} e^{-a p |v|^2}\,dv d\Sigma(z). 
\]

We fix the point $z \in \p \O$ and compute the inner integral with respect to $v$. We introduce the spherical coordinate for $v$ so that the direction $\theta = 0$ corresponds to $-n(z)$. Then, we have
\begin{align*}
\int_{\Gamma^-_z} \frac{1}{N(z, v)^{p-2} |v|^{p-b}} e^{-a p |v|^2}\,dv =& 2\pi \int_0^{\frac{\pi}{2}} \frac{\sin \theta}{\cos^{p-2} \theta}\,d\theta \int_0^\infty \frac{1}{\rho^{p-b-2}} e^{-a p |v|^2}\,d\rho\\
=& 2\pi \int_0^1 t^{2-p}\,dt \int_0^\infty\rho^{2+b-p} e^{-a p \rho^2}\,d\rho.
\end{align*}
The right hand side converges for all $a > 0$ and $b \geq 0$ if and only if $p < 3$ under the restriction $p \geq 1$. This completes the proof since the boundary $\p \O$ is bounded. 
\end{proof}

Let $g$ be the function on $\Gamma^-$ defined in \eqref{inbdry_optimal_convex_3}. Then, by Lemma \ref{lem:tau_dx_ball} and estimates \eqref{Jg_dx_1}-\eqref{Jg_dv_4} with $\varphi_1$ replaced by $\varphi_2$, we see that $\nabla_x Jg, \nabla_v Jg \in L^p_\a(\O \times \R^3)$ for all $0 \leq \a < (1 - \rho)/2$ and $1 \leq p < 3$, and the corresponding solution $f$ to the boundary value problem \eqref{SLBE}-\eqref{inbdry} also belongs to $L^p_\a(\O \times \R^3)$.

In what follows, we show that actually the solution $f$ does not belong to $W^{1, 3}_\a(\O \times \R^3)$ for all $0 \leq \a < (1 - \rho)/2$. To this aim, we assume that the solution belongs to $W^{1, 3}_\a(\O \times \R^3)$, and derive a 
contradiction. We formally differentiate the equation \eqref{integral form} with respect to $x$ to obtain
\begin{equation} \label{eq:dx_ball}
\begin{split}
\nabla_x f(x, v) =& -\nu(v) (\nabla_x \tau(x, v)) Jg(x, v) + (\nabla_x q(x, v)) J (\nabla_X g)(x, v)\\
&+ S_{\O, x} K f(x, v) + S_\O K (\nabla_x f) (x, v).
\end{split}
\end{equation}
By assumption, we see that $S_\O K (\nabla_x f) \in L^3_\a(\O \times \R^3)$ and therefore the integral
\begin{equation} \label{eq:L3_part_of_dx}
\int_{\O \times \R^3} | \nabla_x f - S_\O K \nabla_x f |^3 e^{3 \a |v|^2}\,dxdv
\end{equation}
is bounded. In what follows, we shall prove that the integral \eqref{eq:L3_part_of_dx} is not bounded, which leads to a contradiction.

In what follows, we show in an example that the left term does not belong to $L^3_\a(\O \times \R^3)$, which causes the contradiction. 

Let $r_0 > 0$ and 
\begin{equation} \label{def:Dtr}
\tilde{D}_{\theta_1, r_0} := \{ (x, v) \in \O \times \R^3 \mid q(x, v) \in \partial \O_{\theta_1}, \tau(x, v) \leq 1, |v| < r_0 \}.
\end{equation}
In this region, we have $J (\nabla_X g) = 0$ and 
\begin{align*}
S_{\O, x}Kf (x, v) =& (\nabla_x \tau(x, v)) e^{-\nu(v)\tau(x, v)} \int_{\Gamma^+_{q(x, v)}} k(v, v^*) f(q(x, v), v^*)\,dv^*\\ 
&+ (\nabla_x \tau(x, v)) e^{-\nu(v)\tau(x, v)} \int_{\Gamma^-_{q(x, v)}} k(v, v^*) e^{-\frac{1}{2} |v^*|^2}\,dv^*.
\end{align*}
We substitute the function $f$ in the first term of the right hand side by the integral equation \eqref{integral form} again to obtain
\begin{align*}
\int_{\Gamma^+_{q(x, v)}} k(v, v^*) f(q(x, v), v^*)\,dv^* =& \int_{\Gamma^+_{q(x, v)}} k(v, v^*) Jg(q(q(x, v), v^*), v^*)\,dv^*\\
&+ \int_{\Gamma^+_{q(x, v)}} k(v, v^*) S_\O K f(q(q(x, v), v^*), v^*)\,dv^*.
\end{align*}
Here, from a geometrical observation, we see that $Jg(q(q(x, v), v^*), v^*) \neq 0$ only if $(\pi - \theta_1 - \theta_2)/2 < \theta_{v^*} < \pi/2 + \theta_2$, where $\theta_{v^*}$ is the polar angle of $v^*$. Thus, for any $\epsilon > 0$, we can choose small enough $\theta_2>0$ such that
\[
\left| \int_{\Gamma^+_{q(x, v)}} k(v, v^*) Jg(q(q(x, v), v^*), v^*)\,dv^* \right| < \epsilon
\]
for all $(x, v) \in D_{\theta_1}$. On the other hand, by Lemma \ref{lem:small_second_convex} and \eqref{eq:uniform_est}, we have
\[
\left| \int_{\Gamma^+_{q(x, v)}} k(v, v^*) S_\O K f(q(q(x, v), v^*), v^*)\,dv^* \right| \lesssim \| f \|_{C_\alpha((\O \times \R^3) \cup \Gamma^\pm)} r
\lesssim r.
\]
Therefore, we can make a contribution from the integral
\[
\int_{\Gamma^+_{q(x, v)}} k(v, v^*) f(q(x, v), v^*)\,dv^*
\]
arbitrary small by taking $r$ first then $\theta_2$ sufficiently small.

In what follows, we consider the case $\gamma = 1$, which corresponds to the hard-sphere model. We show the following lemma, which corresponds to Lemma \ref{lem:neg_int_flat}.

\begin{lemma} \label{lem:neg_int_ball}
There exist $\eta_0 > 0$ and $r_0 > 0$ such that 
\begin{equation} \label{ineq:nonvanish_ball}
\nu(v) e^{-\frac{1}{2}|v|^2} - \int_{\Gamma^-_{q(x, v)}} k(v, v^*) e^{-\frac{1}{2} |v^*|^2}\,dv^* > \eta_0
\end{equation}
for all $\tilde{D}_{\theta_1, r_0}$.
\end{lemma}

\begin{proof}
By the continuity of the left hand side of \eqref{ineq:nonvanish_ball}, it suffices to show positivity at the limit $|v| \to 0$. In this limit, we have
\[
\lim_{|v| \to 0} \nu(v) e^{-\frac{1}{2} |v|^2} = \nu(0) = 2^{-\frac{3}{2}} (1 + 1) = 2^{-\frac{1}{2}}
\]
and
\begin{align*}
& \lim_{|v| \to 0} \int_{\Gamma^-_{q(x, v)}} k(v, v^*) e^{-\frac{1}{2}|v^*|^2}\,dv^*\\
=& \int_{\Gamma^-_{q(x, \hat{v})}} k(0, v^*) e^{-\frac{1}{2}|v^*|^2}\,dv^*\\
=& \int_{\Gamma^-_{q(x, \hat{v})}} 2^{-\frac{3}{2}} \pi^{-1} \left\{ 2 |v^*|^{-1} e^{-\frac{1}{2}|v^*|^2} - |v^*| e^{-\frac{1}{2} |v^*|^2 } \right\} e^{-\frac{1}{2}|v^*|^2}\,dv^*
\end{align*}
for a unit vector $\hat{v}$ such that $q(x, \hat{v}) \in \p \O_{\theta_1}$. We remark that the limit does not look appropriate since it formally depends on the choice of the unit vector $\hat{v}$. In what follows, we will see that the limit is independent of it.

We introduce the spherical coordinates for $v^*$ so that the vector corresponding to $\theta = 0$ direct to $-n(q(x, \hat{v}))$. Then, as we computed in the proof of Lemma \ref{lem:neg_int_flat}, we have
\begin{align*}
&\int_{\Gamma^-_{q(x, \hat{v})}} 2^{-\frac{3}{2}} \pi^{-1} \left\{ 2 |v^*|^{-1} e^{-\frac{1}{2}|v^*|^2} - |v^*| e^{-\frac{1}{2} |v^*|^2 } \right\} e^{-\frac{1}{2}|v^*|^2}\,dv^*\\
=& 2^{-\frac{3}{2}} \pi^{-1} \int_0^\infty \left( \int_0^{\pi/2} \left( \int_0^{2\pi} (2 \rho^{-1} - \rho) e^{-\rho^2} \rho^2 \sin \theta\,d\varphi \right)\,d\theta \right)\,d\rho\\
=& 2^{-\frac{1}{2}} \int_0^\infty (2 \rho - \rho^3) e^{-\rho^2} \,d\rho\\
=& 2^{-\frac{3}{2}}.
\end{align*}
Therefore, we have
\[
\nu(0) - \int_{\Gamma^-_{q(x, \hat{v})}} k(0, v^*) e^{-\frac{1}{2}|v^*|^2}\,dv^* = 2^{-\frac{1}{2}} - 2^{-\frac{3}{2}} = 2^{-\frac{3}{2}}. 
\]
This completes the proof.
\end{proof}

For sufficiently small $r, \theta_2$ and $r_0$, thanks to Lemma \ref{lem:neg_int_ball}, we have
\begin{align*}
\int_{\O \times \R^3} | \nabla_x f - S_\O K \nabla_x f |^3 e^{3 \a |v|^2}\,dxdv \gtrsim \int_{\tilde{D}_{\theta_1, r_0}} |\nabla_x \tau(x, v)|^3\,dxdv.
\end{align*}
Here, we perform the same change of variable as \eqref{domain_to_boundary}. We notice again that, because of the restriction $\tau(x, v) < 1$, the integral has a different form:
\begin{align*}
&\int_{\tilde{D}_{\theta_1, r_0}} |\nabla_x \tau(x, v)|^3\,dxdv\\ 
=& \int_{\p \O_{\theta_1}} \int_{\Gamma^-_z \cap \{|v| < r_0\}} \int_0^{\min\{\tau(z, -v), 1\}} |\nabla \tau(z + tv, v)|^3\,dt\, N(z, v) |v|\,dv d\Sigma(z)\\
=& \int_{\p \O_{\theta_1}} \int_{\Gamma^-_z \cap \{|v| < r_0\}} \frac{\min\{\tau(z, -v), 1\}}{N(z, v)^2 |v|^2}\,dv d\Sigma(z).
\end{align*}
We consider the case $\min\{\tau(z, -v), 1\} = \tau(z, -v)$. In this case, from \eqref{eq:tau_ball_ex}, we have $2r N(z,v) \leq |v|$. Thus, we have
\begin{align*}
&\int_{\p \O_{\theta_1}} \int_{\Gamma^-_z \cap \{|v| < r_0\}} \frac{\min\{\tau(z, -v), 1\}}{N(z, v)^2 |v|^2}\,dv d\Sigma(z)\\
\geq& 2r \int_{\p \O_{\theta_1}} \int_{\Gamma^-_z \cap \{ 2r N(z,v) \leq |v| < r_0 \}} \frac{1}{N(z, v) |v|^3}\,dv d\Sigma(z).
\end{align*}

We first perform integration with respect to $v$ variable. We introduce the spherical coordinate system so that $\theta = 0$ direct to the $n(z)$ direction. Without loss of generality, we may assume $r_0 < 2r$. Due to the restriction $2rN(z, v) < r_0$, we have $0 < \cos \theta < r_0/2r$. Introducing the new variable $t = \cos \theta$, we have
\begin{align*}
\int_{\Gamma^-_z \cap \{ 2r N(z,v) \leq |v| < r_0 \}} \frac{1}{N(z, v) |v|^3}\,dv =& 2 \pi \int_0^{r_0/2r} \int_{2rt}^{r_0} \frac{1}{\rho^3 t} \rho^2 \,d\rho dt\\
=& 2 \pi \int_0^{r_0/2r} \frac{\log(r_0) - \log(2rt)}{t}\,dt.
\end{align*}
Since $\log(r_0) - \log(2rt) \to \infty$ as $t \downarrow 0$, the integral in the right hand side diverges to infinity, which implies that the integral \eqref{eq:L3_part_of_dx} is not bounded. This is the contradiction.

%%%%%%%%%%%%%%%%%%%%%%%%%%%%%%%%%%%%%%%%%%%%%%%%%%%%%%
\section*{Acknowledgement}
%%%%%%%%%%%%%%%%%%%%%%%%%%%%%%%%%%%%%%%%%%%%%%%%%%%%%%

 I. Chen is supported in part by NSTC with the grant number 108-2628-M-002-006-MY4, 111-2918-I-002-002-, and 112-2115-M-002-009-MY3. C. Hsia is supported in part by NSTC with the grant number 109-2115-M-002-013-MY3. D. Kawagoe is supported in part by JSPS KAKENHI grant number 20K14344. D. Kawagoe and J. Su is supported in part by the National Taiwan University-Kyoto University Joint Funding.

\end{document}